\documentclass[10pt,oneside,reqno]{amsart}

\usepackage{geometry} 
\usepackage{amsmath, amsthm, amssymb} 
\usepackage{multicol}

\usepackage{hyperref}
\usepackage{array}
\usepackage{tabularx}
\usepackage{lmodern}
\usepackage{enumitem}
\usepackage{epigraph}
\usepackage{mathtools}
\usepackage{longtable} 
\usepackage{listings}
\usepackage{pifont}
\usepackage{array}
\usepackage{lipsum}
\usepackage{scrextend}
\hypersetup{colorlinks=true,linkcolor=blue, linktocpage}
\usepackage{graphicx} 
\usepackage{mathrsfs} 
\usepackage{color}
\usepackage{verbatim}
\usepackage[bbgreekl]{mathbbol}
\usepackage[dvipsnames]{xcolor}
\usepackage{enumitem}
\usepackage{tikz-cd} 
\usepackage{graphicx}
\setlist[enumerate,1]{label=(\arabic*),ref=\arabic*$^\circ$}

\usepackage{geometry}
\makeatletter
\@addtoreset{equation}{section}
\makeatother

\usepackage[all]{xy}
\usepackage{hyperref}
\title{Finite abelian groups acting on rationally connected threefolds II: groups of K3 type}
\author{Konstantin Loginov}

  \address{
\textsc{Steklov Mathematical Institute of Russian Academy of
Sciences, Moscow, Russia \newline Center of Pure Mathematics, MIPT, Moscow, Russia \newline Laboratory of Algebraic Geometry, National Research University Higher School of Economics, Moscow, Russian Federation}}
  \email{loginov@mi-ras.ru}

\author{Antoine Pinardin}
\address{
\textsc{School of Mathematics, University of Edinburgh, UK}
}
\email{antoine.pinardin@ed.ac.uk}

\author{Zhijia Zhang}
\address{
Courant Institute,
  251 Mercer Street,
  New York, NY 10012, USA
}

\email{zz1753@nyu.edu}

\date{} 

\newcounter{cthm}

\newtheorem{proposition}[equation]{Proposition}
\newtheorem{thm}[equation]{Theorem}

\newtheorem{corollary}[equation]{Corollary}
\newtheorem{lem}[equation]{Lemma}
\theoremstyle{definition}
\newtheorem{defin}[equation]{Definition}
\newtheorem{remark}[equation]{Remark}

\newtheorem{conjecture}[equation]{Conjecture}
\newtheorem{question}[equation]{Question}

\newtheorem{example}[equation]{Example}
\theoremstyle{example}

\newcommand{\OOO}{\mathscr{O}}

 \newcommand{\tg}{\{\mathrm{Id}\}}

 \newcommand{\PP}{\mathbb P}

 \newcommand{\Z}{\mathbb Z}
  \newcommand{\C}{\mathbb C}
 
 \newcommand{\aut}{\mathrm{Aut}}
 \newcommand{\gl}{\mathrm{GL}}
 
 \newcommand{\pic}{\mathrm{Pic}}
 \newcommand{\dz}{\mathbb Z}

 \newcommand{\id}{\mathrm{id}}

\let\emptyset\varnothing

\begin{document}

\maketitle

\begin{abstract}
We study finite abelian groups acting on three-dimensional rationally connected varieties. 
We concentrate on the groups of K3  type, that is, abelian extensions by a cyclic group of groups that faithfully act on a K3 surface.   
In particular, if a finite abelian group faithfully acts on a threefold preserving a K3 surface (with at worst du Val singularities), then such a group is of K3 type.
We prove a classification theorem for the groups of K3 type which can act on three-dimensional rationally connected varieties. 
We note the relation between certain groups of K3 type and K3 surfaces with higher Picard number.  
\end{abstract}
\setcounter{tocdepth}{1} 
\tableofcontents

\section{Introduction}
We work over the field of complex numbers $\mathbb{C}$. 
By $\mathrm{Bir}(X)$ we denote the group of birational automorphisms of an algebraic variety $X$. 
We deal with the classification problem of finite subgroups in $\mathrm{Bir}(X)$ when $X$ is a rationally connected variety. More specifically, we are mostly interested in finite subgroups of the Cremona group. 
Recall that the Cremona group is defined as $\mathrm{Cr}_n(\mathbb{C})=\mathrm{Bir}
(\mathbb{P}^n)$. 
The classification of finite subgroups of $\mathrm{Cr}_2(\mathbb{C})$ was obtained in \cite{DI09}. As for finite subgroups of $\mathrm{Cr}_3(\mathbb{C})$, the complete classification seems to be out of reach. There are results concerning some classes of finite groups, see \cite{Pr09} for simple groups. 

In this paper, we concentrate on finite abelian subgroups of $\mathrm{Bir}(X)$ when $X$ is a rationally connected variety. 
The case of $\mathrm{Cr}_1(\mathbb{C})=\mathrm{PGL}(2,\mathbb{C})$ is elementary, see Proposition~\ref{prop-abelian-subgroup-of-pgl2}. 
The classification of finite abelian subgroups in $\mathrm{Cr}_2(\mathbb{C})$ was obtained in \cite{Bl07}, see Theorem \ref{thm-surfaces}. 
The study of finite abelian subgroups in $\mathrm{Cr}_3(\mathbb{C})$ was initiated in \cite{Lo24}. 
In arbitrary dimension, there exists the following result.

\begin{thm}[{\cite[Corollary 11]{KZh24}}]
\label{intro-thm-kollar-zhuang}
Let $X$ be a rationally connected variety of dimension $n$, and let $G\subset \mathrm{Bir}(X)$ be a finite abelian $p$-group. Then $G$ can be generated by $r$ elements where 
\[
r \leq \frac{pn}{p-1}\leq 2n.
\]
\end{thm}
The bound as in Theorem \ref{intro-thm-kollar-zhuang} in dimension $n=2$ was obtained in \cite{Beau07}, it also follows from \cite{Bl07} and \cite{DI09}; 
in dimension~$n=3$ it was obtained in a series of works 
\cite{Pr11}, \cite{Pr14}, \cite{PS17}, \cite{Kuz20},
\cite{Xu20}, \cite{Lo22} using explicit methods of the minimal model program. 

To deal with the case of threefolds, the following definitions were introduced in \cite{Lo24}. 
A finite abelian group $G$ is called \emph{a
group of product type}, if $G= G_1\times G_2$ where $G_i\subset \mathrm{Cr}_i(\mathbb{C})$. 
In particular, $G$ is isomorphic to a subgroup in
\[
\mathrm{Cr}_1(\mathbb{C})\times\mathrm{Cr}_2(\mathbb{C})\subset
\mathrm{Cr}_3(\mathbb{C}).
\] 
Using the classification of finite abelian subgroups of $\mathrm{Cr}_i(\mathbb{C})$ for $i=1,2$, see Proposition \ref{prop-abelian-subgroup-of-pgl2} and Theorem \ref{thm-surfaces}, it is not hard to write down the complete list of groups of product type, see Table $2$ in Section \ref{subsec-pt}. 

We say that a finite abelian group~$G$ is of \emph{K3 type}, if $G$ is an abelian extension of a finite abelian group $H$ that faithfully acts on a K3 surface, by a cyclic group:
\begin{equation}
\label{intro-exact-k3-type}
0\to \mathbb{Z}/m\to G \to H\to 0.    
\end{equation}
In particular, if a finite abelian group faithfully acts on a three-dimensional variety preserving a K3 surface (with at worst du Val singularities), then such a group is of K3 type.

The following theorem is the main result of \cite{Lo24} (see Section \ref{subsec-mfs} for the definitions of a $G\mathbb{Q}$-Mori fiber space and a $G\mathbb{Q}$-Fano variety). 

\begin{thm}[{\cite[Theorem 1.7]{Lo24}}]
\label{thm-first-main-thm}
Let $X$ be a rationally connected variety of dimension~$3$, and let $G\subset \mathrm{Bir}(X)$ be a finite abelian group. Then 
\begin{enumerate}
\item
either $G$ is of product type, 
\item
or $G$ is of K3 type,
\item
or $G$ faithfully acts on a $G\mathbb{Q}$-Fano threefold $X'$ with $|-K_{X'}|=\emptyset$ such that $X'
$ is $G$-birational to $X$. Moreover, any $G\mathbb{Q}$-Mori fiber space with a faithful action of~$G$ is a $G\mathbb{Q}$-Fano threefold with empty anti-canonical system. 
\end{enumerate}
\end{thm}
The three cases in Theorem \ref{thm-first-main-thm} are not mutually exclusive (see Example \ref{ex-k3-pt} below). 
It is known that if a finite abelian group that faithfully acts on a rationally connected threefold preserving a rational curve, a rational surface, or a structure of a Mori fiber space with a non-trivial base, then such a group is of product type, see \cite[Corollary 3.14, Corollary 3.17]{Lo24}. Essentially, this follows from a purely algebraic result on abelian extensions of finite abelian groups, see Proposition \ref{prop-abstact-extension}. It is also known that if a finite abelian group that faithfully acts on a threefold with terminal singularities has a (smooth or singular) fixed point then it is of product type, see Theorem \ref{thm-action-on-terminal-point}. However, not all finite abelian groups that can faithfully act on a rationally connected threefold are of product type. Also, it is expected that there are no groups of type (3) in Theorem \ref{thm-first-main-thm} which are not of product type or of K3 type.

\begin{conjecture}[{\cite[Conjecture 1.8]{Lo24}}]
\label{intro-conjecture}
In the notation of Theorem \ref{thm-first-main-thm}, any group of type (3) 
is either of product type or of K3 type.
\end{conjecture}

In \cite[Corollary 1.10]{Lo24} it is proven that there are only finitely many isomorphism classes of finite abelian groups of K3 type which faithfully act on a rationally connected threefold. 
This result follows from two boundedness results. First, there only finitely many isomorphism classes of finite groups that can faithfully act on a K3 surface, see \cite{BH23} for a complete classification. 
This bounds the group $H$ in the exact sequence~\eqref{intro-exact-k3-type}. 
The second result, which is needed to bound $m$ in \eqref{intro-exact-k3-type}, is the boundedness of the indices of Fano threefolds with canonical singularities. Of course, it follows from the boundedness of Fano threefolds with canonical singularities. Recently, an effective bound equal  to $66$ was obtained in \cite{JL25}. 
Hence the main problem in dealing with groups of K3 type that act on rationally connected threefolds is to bound the number $m$ in \eqref{intro-exact-k3-type} for each group $H$ that can faithfully act on a K3 surface. It turns out that this number indeed can be effectively bounded.


The main result of our work is as follows. 

\begin{thm}
\label{thm-main-thm}
Let $X$ be a rationally connected variety of dimension $3$, and let $G\subset \mathrm{Bir}(X)$ be a finite abelian group. 
Then either $G$ is of product type, or of type (3) as in Theorem~\ref{thm-first-main-thm}, or $G$ is isomorphic to one of the following groups:
\begin{enumerate}
\item 
$(\mathbb{Z}/4)^4$,
\item
$(\mathbb{Z}/6)^3\times \mathbb{Z}/2$,
\item
$(\mathbb{Z}/6)^2\times(\mathbb{Z}/3)^2$,
\item 
$(\mathbb{Z}/8)^2\times\mathbb{Z}/4\times\mathbb{Z}/2$.
\end{enumerate}
\end{thm}

All the cases in Theorem \ref{thm-main-thm} are realized as shown in Example \ref{exam-k3-not-pt}. 
If Conjecture \ref{intro-conjecture} is true, then the list of groups of product type (see Table 2 in Section \ref{subsec-pt}) together with Theorem \ref{thm-main-thm} provide a complete list of finite abelian groups that can act on a rationally connected variety of dimension $3$. 

\begin{corollary}
Let $X$ be a rationally connected variety of dimension $3$, and let $G\subset \mathrm{Bir}(X)$ be a finite abelian group. Assume that  
\begin{itemize}
    \item either Conjecture \ref{intro-conjecture} holds,
    \item 
    or $G$ faithfully acts on a terminal Fano threefold $X'$ with $|-K_{X'}|\neq 0$.
\end{itemize}
Then $G$ is isomorphic to one of the following groups (and all these cases are realized):
\begingroup
\renewcommand{\arraystretch}{1.2}
\begin{center}
\begin{tabular}{ | m{1.7em} | m{5.0cm} | m{3.5cm} | } 
  \hline
   & $G$ &  \\ 
  \hline
  (1) & $\mathbb{Z}/k\times \mathbb{Z}/l\times
\mathbb{Z}/m$ & $k\geq 1,\ l\geq 1,\ m\geq 1$ \\ 
  \hline
  (2) & $\mathbb{Z}/2k\times(\mathbb{Z}/4)^2\times
\mathbb{Z}/2$ & $k\geq 1$ \\ 
    \hline
  (3) & $\mathbb{Z}/3k\times(\mathbb{Z}/3)^3$ & $k\geq 1$ \\ 
    \hline
  (4) & $\mathbb{Z}/2k\times \mathbb{Z}/2l\times
(\mathbb{Z}/2)^2$ & $k\geq 1,\ l\geq 1$ \\ 
    \hline
  (5) & $\mathbb{Z}/2k\times (\mathbb{Z}/2)^4$ & $k\geq 1$ \\ 
    \hline
  (6) & $(\mathbb{Z}/4)^2\times (\mathbb{Z}/2)^3$ & \\ 
    \hline
  (7) & $(\mathbb{Z}/2)^6$ &  \\ 
    \hline
 (8) & $(\mathbb{Z}/4)^4$ &  \\ 
    \hline
     (9) & $(\mathbb{Z}/6)^3\times \mathbb{Z}/2$ &  \\ 
    \hline
     (10) & $(\mathbb{Z}/6)^2\times(\mathbb{Z}/3)^2$ &  \\ 
    \hline
     (11) & $(\mathbb{Z}/8)^2\times\mathbb{Z}/4\times\mathbb{Z}/2$ &  \\ 
    \hline
\end{tabular}

\

{Table 1. Conjectural list of all finite abelian groups that can faithfully act on a rationally connected threefold}
\end{center}
\endgroup
\end{corollary}

In Table 1, the groups (1)--(7) are of product type, so they can act on a rational threefold, while the groups (8)--(11) are of K3 type and not of product type. 

Finite abelian groups of symplectic automorphisms of K3 surfaces were classified by V. Nikulin in the famous paper \cite{Nik80}, see Theorem \ref{thm-symplectic-K3-groups}. 
The classification of Brandhorst and Hofmann \cite{BH23} provides the list of all maximal finite abelian groups that can faithfully act on a K3 surface, cf. Theorem \ref{thm-maximal-k3-groups}. It turns out that all but $6$ of them can be realized as subgroups of $\mathrm{Cr}_2(\mathbb{C})$. 

\begin{proposition}
\label{prop-intro-k3-groups}
Let $H$ be a finite abelian group that faithfully acts on a K3 surface. Assume further that $H$ is not isomorphic to a subgroup of $\mathrm{Cr}_2(\mathbb{C})$, that is, it cannot faithfully act on a rational surface. Then $H$ is isomorphic to one of the following groups:
\begin{enumerate}
\item
$(\mathbb{Z}/4)^3$,
\item
$(\mathbb{Z}/6)^2\times \mathbb{Z}/2$,
\item
$\mathbb{Z}/6\times(\mathbb{Z}/3)^2$,
\item
$\mathbb{Z}/8\times\mathbb{Z}/4\times\mathbb{Z}/2$,
\item
$(\mathbb{Z}/2)^5$,
\item
$\mathbb{Z}/4\times(\mathbb{Z}/2)^3$.
\end{enumerate}
\end{proposition}
Using the results of \cite{BH23}, 
we give a more precise description of the action of these~$6$ groups on K3 surfaces, including the decomposition into symplectic and non-symplectic subgroups and invariant lattice in cohomology of a surface, see Corollary \ref{cor-max-group-direct-product} and Proposition \ref{prop: determinants}.  
In particular, all of these $6$ groups are not symplectic. In fact, all the finite abelian groups of symplectic automorphisms can be realized as subgroups of $\mathrm{Cr}_2(\mathbb{C})$. 

Using the exact sequence \eqref{intro-exact-k3-type} and Proposition \ref{prop-abstact-extension}, we conclude that if $H$ is not one of~$6$ groups from Proposition \ref{prop-intro-k3-groups} then $G$ is of product type. Hence, to study groups of K3 type which are not of product type, we may assume that $H$ is one of the $6$ groups as in Proposition \ref{prop-intro-k3-groups}. The problem is to bound the number $m$ as in \eqref{intro-exact-k3-type}. 

\begin{example}
\label{exam-k3-not-pt}
We construct the actions of groups of K3 type on (singular) Fano varieties which are Fermat hypersurfaces in weighted projective spaces. The corresponding $G$-invariant K3-surfaces are given by the hyperplane sections $x_0=0$.
\begin{enumerate}
\item  
Let $X_4\subset \mathbb{P}^4$ be given by the equation
\[
x_0^4 + x_1^4 + x_2^4 + x_3^4 + x_4^4 = 0,
\]
with the action of $G=(\mathbb{Z}/4)^4$. Note that $X_4$ is smooth.

\item
Let $X_6\subset \mathbb{P}(1,1,1,1,3)$ be 
given by the equation
\[
x_0^6 + x_1^6 + x_2^6 + x_3^6 + x_4^2 = 0,
\]
with the action of $G=(\mathbb{Z}/6)^3\times \mathbb{Z}/2$. Note that $X_6$ is smooth.  

\item
Let $X'_6\subset \mathbb{P}(1,1,1,2,2)$ be 
given by the equation
\[
x_0^6 + x_1^6 + x_2^6 + x_3^3 + x_4^3 = 0,
\]
with the action of  
$G=(\mathbb{Z}/6)^2\times(\mathbb{Z}/3)^2$. Note that $X'_6$ has $3$ singular points of type $\frac{1}{2}\times (1,1,1)$.

\item 
Let $X_8\subset \mathbb{P}(1, 1, 1, 2, 4)$ be given by the equation
\[
x_0^8 + x_1^8 + x_2^8 + x_3^4 + x_4^2 = 0,
\]
with the action of $G=(\mathbb{Z}/8)^2\times\mathbb{Z}/4\times\mathbb{Z}/2$. 
Note that $X_8$ has $2$ singular points of type $\frac{1}{2}\times (1,1,1)$.
\end{enumerate}
\end{example}
\begin{example}
\label{ex-k3-pt}
We give examples of the actions of a finite abelian group $G$ on a rationally connected threefold where $G$ is both of K3 type and of product type.
\label{exam-k3-pt}
\begin{enumerate}
\item[(5)]
Let $X_{2,2,2}\subset \mathbb{P}^6$ be the intersection of three quadrics, so it is given by the equations
\[
\sum_{i=0}^6 x_i^2 = \sum_{i=0}^6 \lambda_i x_i^2 = \sum_{i=0}^6 \mu_i x_i^2 = 0,
\]
with the action of $G=(\mathbb{Z}/2)^6$. 

\item[(6)]
Let $X_{4,4} \subset \mathbb{P}(1, 1, 1, 2, 2, 2)$ be given by the equations
\[
x_0^4 + x_1^4 + x_2^4 + x_3^2 + x_4^2 + x_5^2 = \lambda_0 x_0^4 + \lambda_1 x_1^4 + \lambda_2 x_2^4 + \lambda_3 x_3^2 + \lambda_4 x_4^2 + \lambda_5 x_5^2 = 0, 
\]
with the action of $G=(\mathbb{Z}/4)^2\times (\mathbb{Z}/2)^3$. Note that $X_{4,4}$ has $4$ singular points of type $\frac{1}{2}\times (1,1,1)$.
\end{enumerate}
\end{example}

By \cite[Corollary 1.1.9]{ChP16} the varieties $X$ in cases (1)--(4)
in Example \ref{exam-k3-not-pt} are not rational. 
This observation gives rise to the following question.

\begin{question}
\label{question-exceptional-cremona}
Can the groups (1)--(4) in Example \ref{exam-k3-not-pt} be realized as subgroups of $\mathrm{Cr}_3(\mathbb{C})$? In other words, can such groups faithfully act on a rational threefold? 
\end{question}

If the answer to Question \ref{question-exceptional-cremona} is negative, then Conjecture \ref{intro-conjecture} would imply that all finite abelian subgroups of $\mathrm{Cr}_3(\mathbb{C})$ are of product type.

We note that the Fermat quartic K3 surface $S_4$ in $\mathbb{P}^3$ which is a $G$-invariant hyperplane section of $X_4$ from Example \ref{exam-k3-not-pt}(1) enjoys many nice properties. For example, it has maximal possible Picard rank $20$, see Example \ref{exam-fermat-quartic}. Recall that such K3 surfaces are called \emph{singular} (although later in the text we reserve this term for surfaces having singularities to avoid confusion). In fact $S_4$ is a Kummer K3 surface associated with the product of two isogenous elliptic curves $E_{\sqrt{-1}}$ and $E_{2\sqrt{-1}}$. 
In \cite{Lo24} it is shown that the finite abelian groups in $\mathrm{Cr}_2(\mathbb{C})$ that do not belong to $\mathrm{Cr}_1(\mathbb{C})\times \mathrm{Cr}_1(\mathbb{C})$, that is the groups (3)--(4) from Theorem~\ref{thm-surfaces}, correspond to elliptic curves with complex multiplication. 
As pointed out in \cite{Sch08}, singular K3 surfaces in many ways behave like elliptic curves with complex multiplication. 
Let $S_6, S'_6, S_8, S_{2,2,2}, S_{4,4}$ be the $G$-invariant hyperplane sections given by the equation $x_0=0$ of Fano threefolds (2)--(6) from Example \ref{exam-k3-not-pt} and Example \ref{exam-k3-pt}. 
Using the results of \cite{EL25}, one computes $\rho(S_6)=\rho(S'_6)=20$, $\rho(S_8)=18$. It is known that for a Kummer K3 surface, its Picard rank is greater or equal than $17$. 
This observation motivates the following question.

\begin{question}
Are $S_6, S'_6, S_8, S_{2,2,2}, S_{4,4}$ Kummer K3 surfaces?
\end{question}

It is known that singular K3 surfaces are classified by its transcendental lattice $T_S$ which is even, positive definite and has rank $2$. If the values of the quadratic form on this lattice are divisible by~$4$, then $S$ is Kummer. Similar criteria are known for K3 surfaces with Picard rank at least $17$, cf. \cite[14.3.20]{Hu16}. It would be interesting to compute transcendental lattices for the above surfaces. 

\subsection*{Sketch of proof of Theorem \ref{thm-main-thm}}
Let $G$ be a group that faithfully acts on a rational connected threefold $X$. Using $G$-equivariant resolution of singularities and $G$-equivariant minimal model program, we may assume that $X$ is a projective $G\mathbb{Q}$-Mori fiber space over the base $Z$. If $\dim Z>0$ then $G$ is of product type by \cite[Corollary 3.17]{Lo24}. Hence we may assume that $X$ is a $G\mathbb{Q}$-Fano threefold. 

By assumption, $G$ is not of type (3) as in Theorem \ref{thm-first-main-thm}, hence we may assume that $h^0(-K_X)>0$. Also, by the proof of \cite[Theorem 1.7]{Lo24} we may assume that for any $G$-invariant element $S\in |-K_X|$, 
the surface $S$ is a K3 surface with at worst du Val singularities. 
In the notation of the exact sequence \eqref{intro-exact-k3-type}, we need to bound $m$ for each finite abelian group $H\subset \mathrm{Aut}(S)$. 
If $H\subset \mathrm{Cr}_2(\mathbb{C})$, then by Proposition \ref{prop-abstact-extension} the group $G$ is of product type. Thus we may assume that $H$ is not isomorphic to a subgroup of $\mathrm{Cr}_2(\mathbb{C})$. Using the classification of \cite{BH23}, cf. Theorem \ref{thm-maximal-k3-groups} and Remark \ref{thm-maximal-k3-groups}, we see that there exist only $6$ possibilities for~$H$. 

We distinguish between two different cases: $h^0(-K_X)\geq 2$ and $h^0(-K_X)=1$. 
First, we deal with the more simple case $h^0(-K_X)\geq 2$. We pick two $G$-invariant K3 surfaces $S, S'\in |-K_X|$. Lemma \ref{lem-C1C2G1G2} shows that $\mathbb{Z}/m$ as in \eqref{intro-exact-k3-type} embeds into $H'\subset \mathrm{Aut}(S')$, where $H'$ is also one of the groups (1)--(6) from Theorem \ref{thm-maximal-k3-groups}. This allows us to bound $m$ in this case. More precise analysis (see Theorem \ref{thm-h0-geq-2}) shows that either $G$ is of product type, or~$G$ is isomorphic to one of the $4$ exceptional groups (1)--(4) as in Theorem \ref{thm-main-thm}. 

It remains to deal with the case $h^0(-K_X)=1$. Here we have only one element $S\in|-K_X|$ which is automatically $G$-invariant. The orbifold Riemann-Roch formula \eqref{eq-ORR} implies that in this case the set of non-Gorenstein singularities of $X$ is non-empty. We describe the possible $G$-orbits of such points depending on the group $H$ in Lemmas~{\ref{lem-exclusion-of-two-groups}--\ref{Z4222 orbits}}. This allows us to list all possibilities for the basket of singularities (cf. Section~\ref{sec-about-baskets}) of $X$ in Proposition \ref{prop: 91 poss}. 

Then, we use local analysis of the terminal singularities on $X$ as well as some global geometric data on the K3 surface $S$ to conclude. 
In particular, we use the description given in \cite{BH23} of the invariant cohomology group of $S$ which is given in Proposition~\ref{prop: determinants}, see also Corollary~\ref{cor-rank-1-is-smooth} and Corollary~\ref{cor-rank-2-is-transitive} for geometric implications. We also need the relation between terminal singularity of $X$ and the corresponding singular point on $S$, see Section \ref{subsec-action-on-terminal}. Finally, we need various properties of symplectic and non-symplectic automorphisms of K3 surfaces as explained in Section~\ref{sect-action-on-K3}. 

Given all this, we proceed to analyze the action of $G$ on $X$ starting from the case where $X$ has only points of type $\frac{1}{2}(1,1,1)$ as non-Gorenstein singularities. Then we consider the case of more general terminal cylic quotient singularities. Finally, we deal with the case of arbitrary terminal singular points. 
It turns out that in all the cases either $G$ is of product type, of $G$ is isomorphic to one of the four exceptional groups as in Theorem \ref{thm-main-thm}. This concludes the proof. 

\subsection*{Structure of the paper}
In Section \ref{sec-prelim}, we collect some preliminary results. 
In particular, we discuss extensions of finite abelian groups and known results on finite abelian subgroups of Cremona groups in dimensions $1$ and $2$. 
In Section \ref{sect-action-on-K3}, we study group actions on K3 surfaces with the emphasis on the classification of finite abelian groups acting on K3 surfaces. 
In Section \ref{sec-K3-lattices}, we consider lattices in the cohomology group of a K3 surface. 
In Section \ref{sec-terminal-sing-classification}, we recall the classification of three-dimensional terminal singularities, study their geometry and group actions on them. 
In Section \ref{sec-h0-geq-2}, we treat the case $h^0(-K_X)\geq 2$ which turns out to be rather elementary. The rest of the paper is devoted to the more complicated case $h^0(-K_X)=1$. 
In Section~\ref{sec-orbits}, we describe possible orbits of non-Gorenstein singularities depending on the group $H$. 
In Section \ref{sec-Case $h^0(-K_X)=1$}, we consider the case where all the non-Gorenstein singularities of $X$ are the points of type $\frac{1}{2}(1,1,1)$. 
In Section \ref{sec-cyclic-quotient}, we treat the case when all the non-Gorenstein singularities of $X$ are cyclic quotient singularities. 
In Section \ref{sec-terminal-sing}, we treat the case of terminal points which are non necessarily cyclic quotient singularities. 

\subsection*{Acknowledgements}
The work of the first author was performed at the Steklov International Mathematical Center and supported by the Ministry of Science and Higher Education of the Russian Federation (agreement no. 075-15-2022-265), supported by the HSE University Basic Research Program, the Simons Foundation, and by the state assignment of MIPT (project FSMG-2023-0013). The first author is a Young Russian Mathematics award winner and would like to thank its sponsors and jury. 
The second author achieved his work at the University of Edinburgh, and was supported by the Leverhulme
Trust grant RPG-2021-229.
The authors thank Ivan Cheltsov, Dmitri Orlov, and Yuri Prokhorov for helpful discussions, and Alexander Kuznetsov for reading the draft of the paper and many useful remarks.

\section{Preliminaries}
\label{sec-prelim}
We work over the field of complex numbers $\mathbb{C}$. 
All varieties are
projective and defined over~$\mathbb{C}$ unless  stated otherwise. We will 
use the language of the minimal model program (the MMP for short), see
e.g. \cite{KM98}.

\subsection{Group actions}
We start with the following well-known results. 
\begin{lem}[cf. {\cite[Lemma 4]{Po14}}]
\label{lem-faithful-action}
Let $X$ be an algebraic variety, and $G\subset \mathrm{Aut}(X)$ be a finite subgroup. Assume $P\in X$ is a fixed point of $G$.  
Then the induced action of~$G$ on the tangent space $T_P X$ is faithful.  
\end{lem}

By $\mathfrak{r}(G)$ we denote the \emph{rank} of a group $G$, that is, the minimal number of generators. 

\begin{example}
If a finite abelian group $G$ linearly and faithfully acts on a vector space~$V$ then $\mathfrak{r}(G)\leq \dim V$. 
\end{example}

\begin{lem}[{cf. \cite[Lemma 2.6]{Pr11}, \cite[Lemma 2.8]{Lo24}}]
\label{cor-fixed-curve-surface}
Let $X$ be a three-dimensional algebraic variety with isolated singularities, and $G\subset \mathrm{Aut}(X)$ be a finite abelian subgroup. 
\begin{enumerate}
\item 
If there is a curve $C\subset X$ of $G$-fixed points, then $\mathfrak{r}(G) \leq 2$. 
\item
If there is a (possibly reducible) divisor $S\subset X$ of $G$-fixed points, then $\mathfrak{r}(G)\leq 1$.
If moreover $S$ is singular along a curve, then $G$ is trivial.
\item 
If $X$ is smooth, and $S\subset X$ is a Weil divisor of $G$-fixed points such that $S$ is singular, then $G$ is trivial. 
\end{enumerate}
\end{lem}

\subsection{Extensions of finite abelian groups}

Let $G$ be a finite abelian group.
In what follows, we will denote by $G_p$ the $p$-Sylow subgroup of $G$ where $p$ is a prime number, so we have
\[
G=\prod_{p} G_p.
\]
We say that $G_p$ is the \emph{$p$-part} of $G$.  Also, a sequence of finite abelian groups 
\begin{equation}
\label{extension-just-another-extension}
0\to H\to G\to K \to0
\end{equation}
is exact if and only if for any prime $p$ the $p$-parts $H_p, G_p, K_p$ of the groups $H,G,K$, respectively, form an exact sequence
\begin{equation}
\label{extension-p-groups}
0\to H_p\to G_p\to K_p \to0.
\end{equation}
We say that the exact sequence \eqref{extension-p-groups} is the \emph{$p$-part} of the exact sequence \eqref{extension-just-another-extension}.

For an abelian $p$-group $G_p$, we say that $G_p$ has type 
\[
\lambda=[\lambda_1,\ldots,\lambda_k] \quad \quad \quad \text{for} \quad \quad \quad \lambda_1\geq\ldots\geq \lambda_k > 0
\]
if 
\[
G_p=\mathbb{Z}/p^{\lambda_1}\times\ldots\times\mathbb{Z}/p^{\lambda_k}.
\]
Note that the type of an abelian $p$-group is defined uniquely. 

To any type $\lambda=[\lambda_1,\ldots,\lambda_k]$ corresponds the Young diagram with $\lambda_i$ boxes in the $i$-th row. For two Young diagrams $\lambda=[\lambda_1,\ldots,\lambda_k]$ and $\mu=[\mu_1,\ldots \mu_l]$, one can define their product $\lambda\cdot \mu$ as a formal linear combination of Young diagrams with non-negative coefficients, see e.g. {\cite[Section 2]{Fu00}}. Then the \emph{Littlewood--Richardson coefficient} $c^\nu_{\lambda\mu}$ is the coefficient at the Young diagram $\nu=[\nu_1,\ldots \nu_s]$ in the product of Young diagrams~$\lambda\cdot \mu$. 

We recall the following criterion, which gives all the possible isomorphism classes of groups which fit into an exact sequence \eqref{extension-p-groups}, in the case of finite abelian $p$-groups.

\begin{thm}[{\cite[Section 2]{Fu00}}]
\label{littelwood}
Let $G_p,H_p,$ and $K_p$ be finite abelian $p$-groups, respectively of types $\mu,\lambda$ and $\nu$. 
Then an extension of the form
\[
	1\to H_p\to G_p\to K_p\to1
\]
exists if and only if for the Littlewood--Richardson coefficient we have $c^\mu_{\lambda\nu}>0$. 
\end{thm}

\subsection{Groups of product type}
\label{subsec-pt}
We recall the results on finite abelian subgroups of Cremona groups in lower dimensions. The one-dimensional case is elementary.
\begin{proposition}
\label{prop-abelian-subgroup-of-pgl2}
Let $G$ be a finite abelian subgroup of $\mathrm{Cr}_1(\mathbb{C})=\mathrm{Aut}(\mathbb{P}^1)=\mathrm{PGL}(2, \mathbb{C})$. Then $G$ is isomorphic to one of the following groups: 
\begin{enumerate}
\item
$\mathbb{Z}/n$,\ \ $n\geq1$,
\item
$(\mathbb{Z}/2)^2$. 
\end{enumerate}
\end{proposition}

\begin{thm}[{\cite{Bl07}}]
\label{thm-surfaces}
Let $G$ be a finite abelian subgroup of $\mathrm{Cr}_2(\mathbb{C})$. Then $G$ is isomorphic to one of the following groups:
\begin{enumerate}
\item
\label{cremona-plane-abelian-1}
$\mathbb{Z}/n\times \mathbb{Z}/m$,\ \ $n\geq
1,\ m\geq 1$,

\item
\label{cremona-plane-abelian-2}
$\mathbb{Z}/2n\times (\mathbb{Z}/2)^2$,\ \ $n\geq 1$,

\item
\label{cremona-plane-abelian-3}
$(\mathbb{Z}/4)^2\times \mathbb{Z}/2$,

\item
\label{cremona-plane-abelian-4}
$(\mathbb{Z}/3)^3$,

\item
\label{cremona-plane-abelian-5}
$(\mathbb{Z}/2)^4$.
\end{enumerate}
\end{thm}

\begin{defin}
\label{def-product-type}
We say that a finite abelian group $G$ is \emph{a
group of product type} if~$G= G_1\times G_2$ where $G_i\subset \mathrm{Cr}_i(\mathbb{C})$. 
In particular, $G$ is isomorphic to a subgroup in
\[
\mathrm{Cr}_1(\mathbb{C})\times\mathrm{Cr}_2(\mathbb{C})\subset
\mathrm{Cr}_3(\mathbb{C}).
\] 
\end{defin}

For example, if for a finite abelian group $G$ we have $\mathfrak{r}(G)\leq 3$, then $G$ is of product type.  
Using Proposition \ref{prop-abelian-subgroup-of-pgl2} and Theorem \ref{thm-surfaces}, we obtain the list of groups of product type:
\begingroup
\renewcommand{\arraystretch}{1.2}
\begin{center}
\label{table-pt}
\begin{tabular}{ | m{1.3em} | m{5.0cm} | m{3.5cm} | } 
  \hline
   & $G$ &  \\ 
  \hline
  (1) & $\mathbb{Z}/k\times \mathbb{Z}/l\times
\mathbb{Z}/m$ & $k\geq 1,\ l\geq 1,\ m\geq 1$ \\ 
  \hline
  (2) & $\mathbb{Z}/2k\times(\mathbb{Z}/4)^2\times
\mathbb{Z}/2$ & $k\geq 1$ \\ 
    \hline
  (3) & $\mathbb{Z}/3k\times(\mathbb{Z}/3)^3$ & $k\geq 1$ \\ 
    \hline
  (4) & $\mathbb{Z}/2k\times \mathbb{Z}/2l\times
(\mathbb{Z}/2)^2$ & $k\geq 1,\ l\geq 1$ \\ 
    \hline
  (5) & $\mathbb{Z}/2n\times (\mathbb{Z}/2)^4$ & $n\geq 1$ \\ 
    \hline
  (6) & $(\mathbb{Z}/4)^2\times (\mathbb{Z}/2)^3$ & \\ 
    \hline
  (7) & $(\mathbb{Z}/2)^6$ &  \\ 
    \hline
\end{tabular}

\
\par
\emph{Table 2. Groups of product type}
\end{center}
\endgroup

The next proposition establishes the main property of finite abelian subgroups in the Cremona groups of rank $1$ and $2$. 

\begin{proposition}[{\cite[Proposition 3.12]{Lo24}}]
\label{prop-abstact-extension}
Let $H\subset \mathrm{Cr}_1(\mathbb C)$ and $K\subset \mathrm{Cr}_2(\mathbb C)$ be finite abelian groups. Then an abelian extension $G$ of $H$ by $K$ (or $K$ by $H$) is of product type.
\end{proposition}

The main object of study in this paper is the following class of groups. 

\begin{defin}
    We say that a finite abelian group~$G$ is of \emph{K3 type}, if $G$ is an abelian extension of a finite abelian group $H$ that faithfully acts on a K3 surface, by a cyclic group:
\begin{equation}
\label{eq-k3-type-prelim}
0\to \mathbb{Z}/m\to G \to H\to 0.
\end{equation}
\end{defin}
\begin{remark}
\label{rem-k3-or-product}
If the group $H$ from \eqref{eq-k3-type-prelim} is isomorphic to a subgroup of $\mathrm{Cr}_2(\mathbb{C})$, then by Proposition \ref{prop-abstact-extension} wee see that $G$ is of product type. Hence, to study groups of K3 type which are not of product type, we may assume that $H$ is not isomorphic to one of the groups from Theorem \ref{thm-surfaces}.
\end{remark}

\subsection{Fermat complete intersections}
The main source of examples of the group actions on algebraic varieties will be given by the following construction. 
Consider a weighted projective space 
\[
\PP=\PP(a_0^{r_0}:\ldots:a_M^{r_M})
\]
where $a_i^{r_i}$ stands for $r_i\geq 1$ consecutive identical weights $a_i$, and $1\leq a_0\leq\ldots\leq a_M$. Put $N=\sum r_i$.
Then $\PP$ is called well-formed, if $\gcd(a_i)=1$ for any set of $N-1$ numbers $a_i$. Let us recall the structure of the automorphism groups of weighted projective spaces, cf. \cite{PrS20}.

\begin{lem}
\label{lem-automoprhism-of-wps}
Let $\PP=\PP(a_0^{r_0},\dots,a_M^{r_M})$ be a well-formed weighted projective space, 
and $1\leq a_0\leq\ldots\leq a_M$. Then $\aut(P)=R\rtimes L$, where $R$ is generated by automorphisms of the form
\[\begin{tikzcd}
	{[x_{0,1}:\ldots:x_{0,r_0}:x_{1,1}:\ldots:x_{1,r_1}:\ldots:x_{M,1}:\ldots:x_{M,r_M}]} \\
	{        [x_{0,1}:\ldots:x_{0,r_0}:x_{1,1}+\phi_{1,1}:\ldots:x_{1,r_1}+\phi_{1,r_1}:\ldots:x_{M,1}+\phi_{M,1}:\ldots:x_{M,r_M}+\phi_{M,r_M}],}
	\arrow[maps to, from=1-1, to=2-1]
\end{tikzcd}\]
    where each $\phi_{p,q}$ is a polynomial of degree $a_i$ in the variables $x_{i,j}$ with $i<p$ and $1\leq j\leq r_i$,
    and $L$ is the quotient of $\gl_{r_0}\times\dots\times\gl_{r_M}$ by $\{(t^{a_0}I_{r_0},\dots,t^{a_M}I_{r_M}),t\in\C^\times\}=\C^\times$.
\end{lem}

\begin{defin}
We define a \emph{Fermat hypersurface} of degree $d$ in a well-formed weighted projective space $\PP=\PP(a_0^{r_0}:\dots:a_M^{r_M})$ as follows:
\[
X_d=\left\{ \sum x_{i,j}^{d/a_i}=0 \right\} \subset \PP,
\]
where $d$ is divisible by $a_i$ for any $i$. 

Similarly, a \emph{Fermat complete intersection} of multidegree $d_1\cdot\ldots\cdot d_k$ for $k\geq 1$ in a  well-formed weighted projective space $\PP$ is given by
\[
X_{d_1\cdot\ldots\cdot d_k} =\left\{ \sum \lambda_{i,j;1} x_{i,j}^{d_1/a_i} = \ldots = \sum \lambda_{i,j;k} x_{i,j}^{d_k/a_i}=0\right\} \subset \PP,
\]
where $d_s$ is divisible by $a_i$ for any $s$ and any $i$,  and $\lambda_{i,j;s}\in \mathbb{C}$. 
\end{defin}

\begin{remark}
Note that a Fermat hypersurface is a (singular) Fano variety if and only $d<\sum_{i=0}^M r_ia_i$. A Fermat complete intersection is a (singular) Fano variety if and only if $\sum_{i=0}^k d_i<\sum_{i=0}^M r_ia_i$. 
\end{remark}

\begin{lem}
\label{lem-group-fermat-aut}
Let $X=X_d\subset \PP=\PP(a_0^{r_0}:\dots:a_M^{r_M})$ be a Fermat hypersurface. Then the group 
\[
G = ((\mathbb{Z}/(d/a_0))^{r_0}\times\ldots\times(\mathbb{Z}/(d/a_M))^{r_M})/(\mathbb{Z}/d)
\]
faithfully acts on $X$. 

Let $X'=X_{d_1\ldots d_k}\subset \PP=\PP(a_0^{r_0}:\dots:a_M^{r_M})$ be a Fermat complete intersection. Put $d'=\gcd(d_s)_{1\leq s\leq k}$. Then the group 
\[
G = ((\mathbb{Z}/(d'/a_0))^{r_0}\times\ldots\times(\mathbb{Z}/(d'/a_M))^{r_0})/(\mathbb{Z}/d')
\]
faithfully acts on $X'$.  
\end{lem}
\begin{proof}
Follows from Lemma \ref{lem-automoprhism-of-wps}.
\end{proof}

\subsection{Lattices}
We recall some generalities on lattices, see e.g. \cite[Chapter 14]{Hu16} or \cite{Nik79}. 
By a lattice $\Lambda$ we mean a free finitely generated abelian group $\mathbb{Z}^n$ equipped with a symmetric bilinear form 
\[
B_\Lambda: \Lambda\times \Lambda\to \mathbb{Z}.
\]
The lattice $\Lambda$ is called even if $Q_\Lambda(v):=B_\Lambda(v,v)$ is even for any $v\in \Lambda$. 

The dual lattice $\Lambda^*$ is defined as 
\[
\Lambda^* = \{v\in \Lambda\otimes \mathbb{Q}\,|\, B_\Lambda(v, w)\in \mathbb{Z}\ \text{for any}\ w\in \Lambda\},
\]
where the bilinear form $B_\Lambda$ is naturally extended to $\Lambda\otimes \mathbb{Q}$. 
\emph{The discriminant group} of $\Lambda$ is defined as 
\[
A_\Lambda=\Lambda^*/\Lambda.
\]
If $B_\Lambda$ is non-degenerate then $A_\Lambda$ is a finite abelian group. In this case, its order $\mathrm{disc}(\Lambda)=|A_\Lambda|$ is called \emph{the discriminant} of $\Lambda$. Note that 
\[
\mathrm{disc}(\Lambda)=|\det B_\Lambda|,
\]    
where by abuse of notation we denote by $B_\Lambda$ the Gram matrix of $\Lambda$.

In what follows, by $k\Lambda$ we will denote the lattice obtained from a lattice $\Lambda$ by multiplying all its vectors by $k\in\mathbb{Z}$. In particular, we have $B_{k\Lambda}=k^2B_{\Lambda}$.

\subsection{Mori fiber spaces}
\label{subsec-mfs}
Let $G$ be a finite group. 
Recall that a normal projective $G$-variety~$X$ is called $G\mathbb{Q}$-factorial, if every $G$-invariant Weil divisor on~$X$ is $\mathbb{Q}$-Cartier. 
A $G\mathbb{Q}$-\emph{Mori fiber space} is a $G\mathbb{Q}$-factorial
variety $X$ 
with at worst terminal singularities together with a
$G$-equivariant contraction $f\colon X\to Z$ to a normal variety $Z$ such that $\rho^G(X/Z)=1$
and
$-K_X$ is ample over $Z$. 
If $Z$ is a point, we say that $X$ is a $G\mathbb{Q}$-Fano variety. For more details on $G$-Mori fiber spaces we refer to \cite{Pr21}.

\section{Group action on K3 surfaces}
\label{sect-action-on-K3}
In this section, we study actions of finite abelian groups on K3 surfaces. By a K3 surface we mean a normal projective surface
$S$ with at worst canonical (that is, du Val) singularities such that
$\mathrm{H}^1(S, \OOO_S)=0$ and $K_S$ is linearly trivial. 

Let
$H\subset\mathrm{Aut}(S)$ be a finite group where $S$ is a smooth projective K3 surface (we always may assume that $S$ is smooth by passing to the minimal resolution). There is a natural
exact sequence (cf. \cite{Hu16})
\begin{equation}
\label{K3-exact-sequence}
0\to H_s \to H \xrightarrow{\alpha} \mathbb{Z}/m \to 0,
\end{equation}
where $\mathbb{Z}/m$ is a cyclic group that acts via
multiplication by a primitive $m$-th root of unity on a non-zero holomorphic
$2$-form $\omega_S$ on $S$.

\begin{defin}
Let $\sigma$ be a finite order automorphism of a K3 surface $S$. Then it
is called a \emph{symplectic automorphism}, if $\alpha(\sigma)=1$.
Otherwise, it is called \emph{non-symplectic}. Moreover, we call
$\sigma$ \emph{purely non-symplectic}, if $\mathrm{ord}
(\alpha(\sigma)) = \mathrm{ord} (\sigma)$. 

A group $H$ acting on $S$ is called \emph{symplectic} (respectively, \emph{non-symplecic}, \emph{purely non-symplectic}), if every non-trivial element of $H$ is symplectic (respectively, non-symplectic, purely non-symplectic). 
\end{defin}

Recall the following classical result.

\begin{thm}[{\cite[4.5]{Nik80}}]
\label{thm-symplectic-K3-groups}
In the exact sequence \eqref{K3-exact-sequence}, if the group $H_s$ is abelian, then it 
is isomorphic to one of the following groups:
\begin{enumerate}
\item
\label{K3-case-1}
$\mathbb{Z}/n$, $1\leq n\leq 8$,
\item
\label{K3-case-2}
$\mathbb{Z}/2\times\mathbb{Z}/6$,
\item
\label{K3-case-3}
$(\mathbb{Z}/3)^2$,
\item
\label{K3-case-4}
$(\mathbb{Z}/4)^2$,
\item
\label{K3-case-5}
$\mathbb{Z}/2\times\mathbb{Z}/4$,
\item
\label{K3-case-6}
$(\mathbb{Z}/2)^k$, $1\leq k\leq 4$.
\end{enumerate}
\end{thm}

Symplectic automorphisms satisfy many nice properties. For example, there is good control on the number of their fixed points.

\begin{proposition}[{cf. \cite[15.1.8]{Hu16}}]
\label{prop-sympl-fixed-points}
A symplectic automorphism of finite order on a smooth K3 surface has finitely
many fixed points. More precisely, for such an automorphism $\sigma$, if we denote by $\mathrm{Fix}(\sigma)$ its fixed locus, we 
have
\begingroup
\renewcommand*{\arraystretch}{1.2}
\begin{longtable}{|c|c|c|c|c|c|c|c|}
\hline
$\mathrm{ord} (\sigma)$ & $2$ & $3$ & $4$ & $5$ & $6$ & $7$ & $8$ \\ \hline
$|\mathrm{Fix}(\sigma)|$ & $8$ & $6$ & $4$ & $4$ & $2$ & $3$ & $2$ \\ \hline
\end{longtable}
\endgroup
\end{proposition}

\begin{remark}
\label{rem-nikulin-involution}
A \emph{Nikulin involution} is a symplectic automorphism of order~$2$ on a smooth K3 surface. According to Proposition \ref{prop-sympl-fixed-points}, the fixed locus of a Nikulin involution consists of exactly $8$ points.
\end{remark}

We denote by $\rho(S)$ the Picard rank of a (smooth) K3 surface $S$. It is well-known that $1\leq \rho(S)\leq 20$.

\begin{proposition}[{cf. \cite[15.1.14]{Hu16}}]
\label{K3-Euler-function}
In the exact sequence \eqref{K3-exact-sequence}, the number $m$ satisfies $m\leq 66$.
\end{proposition}
Also, if a (purely) non-symplectic automorphism $\sigma$ 
has order $m$, the list of all possibilities for $m$ is given in \cite[Corollary 1.3]{BH23}. For example, if $m$ is prime, then $m\in \{2,3,5,7,11,13,17,19\}$.

The next theorem is a generalization of Theorem \ref{thm-symplectic-K3-groups} to the case of non-symplectic groups.   
\begin{thm}[{\cite{BH23}}]
\label{thm-maximal-k3-groups}
Let $H$ be a finite abelian group that faithfully acts on a smooth K3 surface. Assume that $H$ is not purely non-symplectic. 
If $H$ is assumed to be maximal with these properties, then $H$ is isomorphic to one of the following groups:
\begin{multicols}{2}
\begin{enumerate}
\item
$(\mathbb{Z}/4)^3$,
\item
$(\mathbb{Z}/6)^2\times \mathbb{Z}/2$,
\item
$\mathbb{Z}/6\times(\mathbb{Z}/3)^2$,
\item
$\mathbb{Z}/8\times\mathbb{Z}/4\times\mathbb{Z}/2$,
\item
$(\mathbb{Z}/2)^5$,
\item
$\mathbb{Z}/4\times(\mathbb{Z}/2)^3$,
\item
$\mathbb{Z}/12\times\mathbb{Z}/6$,
\item
$\mathbb{Z}/60$,
\item
$\mathbb{Z}/10\times\mathbb{Z}/5$,
\item
$\mathbb{Z}/12\times\mathbb{Z}/4$,
\item
$\mathbb{Z}/12\times(\mathbb{Z}/2)^2$,
\item 
$\mathbb{Z}/18\times\mathbb{Z}/3$,
\item 
$\mathbb{Z}/15\times\mathbb{Z}/3$,
\item 
$\mathbb{Z}/42$,
\item 
$\mathbb{Z}/30\times\mathbb{Z}/2$,
\item 
$\mathbb{Z}/28\times\mathbb{Z}/2$,
\item 
$\mathbb{Z}/24\times\mathbb{Z}/2$,
\item 
$\mathbb{Z}/20\times\mathbb{Z}/2$,
\item 
$\mathbb{Z}/18\times\mathbb{Z}/2$,
\item 
$\mathbb{Z}/16\times\mathbb{Z}/2$.
\end{enumerate}
\end{multicols}
\begin{proof}
    This comes from the full classification of such groups in \cite{BH23}. We list all groups and the code to find maximal ones in \cite{Z-web}.
\end{proof}
\end{thm}

\begin{remark}
\label{rem-special-K3-groups-are-PT}
Using Theorem \ref{thm-surfaces}, one checks that among all the groups in Theorem \ref{thm-maximal-k3-groups}, only the groups (1)--(6) are not isomorphic to a subgroup of $\mathrm{Cr}_2(\mathbb{C})$. 
However, all the groups in Theorem \ref{thm-maximal-k3-groups} are isomorphic to subgroups of $\mathrm{Cr}_3(\mathbb{C})$, and moreover, they are all of product type, see Table 2.
\end{remark}

\begin{corollary}
\label{cor-max-group-direct-product}
Let $H$ be one of the groups (1)--(6) from Theorem \ref{thm-maximal-k3-groups}. Then the exact sequence \eqref{K3-exact-sequence} splits, so we have
\[
H = H_s \times \mathbb{Z}/m
\]
where $H_s$ is the subgroup of symplectic automorphisms, and $\mathbb{Z}/m$ is purely non-symplectic. 
More precisely, one of the following holds:
\begin{enumerate}
    \item 
    $H_s=(\mathbb{Z}/4)^2,\quad m=4$,
    \item 
    $H_s=\mathbb{Z}/6\times\mathbb{Z}/2,\quad m=6$,
    \item 
    $H_s=(\mathbb{Z}/3)^2,\quad m=6$,
    \item 
    $H_s=\mathbb{Z}/4\times\mathbb{Z}/2,\quad m=8$,
    \item 
    $H_s=(\mathbb{Z}/2)^5,\quad m=2$,
    \item 
    $H_s = (\mathbb{Z}/2)^3,\quad m=4$.

\end{enumerate}
\end{corollary}
\begin{proof}
From the case by case analysis using Theorem \ref{thm-maximal-k3-groups}, Theorem \ref{thm-symplectic-K3-groups} and the exact sequence \eqref{K3-exact-sequence} we obtain the cases (1)--(6) and one exceptional case $H_s = (\mathbb{Z}/2)^4, m=2$. However, the latter case is not realized according to \cite{BH23}.
\end{proof}

We collect some examples of actions of finite abelian groups on K3
surfaces. 

\begin{example}
\label{example-smooth-k3-action}
We start with the case of smooth K3 surfaces which  are Fermat complete intersections on which the following group acts faithfully (cf. Lemma \ref{lem-group-fermat-aut}).
\begingroup
\renewcommand{\arraystretch}{1.2}
\begin{center}
\begin{longtable}{|p{0.03\textwidth}|p{0.35\textwidth}|p{0.35\textwidth}|}
\hline
 & \textbf{K3 surface} & \textbf{Group} \\ \hline
\endhead
(1) & $X_{6} \subset \mathbb{P}(1,1,1,3)$ & $(\mathbb{Z}/6)^2\times
\mathbb{Z}/2$ \\ \hline
(2) & $X_{4} \subset \mathbb{P}^3$ & $(\mathbb{Z}/4)^3$ \\ \hline
(3) & $X_{2, 2, 2} \subset \mathbb{P}^5$ & $(\mathbb{Z}/2)^5$ \\ \hline
\end{longtable}
\emph{Table 3. Examples of smooth K3 surfaces with the action of a finite abelian group}
\end{center}
\endgroup
\end{example}

\begin{example}
\label{exam-singular-k3}
Now we treat the case of singular K3 surfaces, cf. \cite{I-F00}. Taking the minimal resolution of a du Val K3 surface $S$, we
obtain a smooth K3 surface $S'$ with the action of the group $H$. We also describe singularities of $S$. For example, we write $\mathrm{Sing}(S)=3A_1+4A_2$, if $X$ has $3$ du Val singular points of type $A_1$ and $4$ singular points of type $A_2$. To compute the group~$H$, one can use Lemma \ref{lem-group-fermat-aut}.  
\begingroup
\renewcommand{\arraystretch}{1.2}
\begin{center}
\begin{longtable}{|p{0.03\textwidth}|p{0.35\textwidth}|p{0.35\textwidth}|p{0.16\textwidth}|}
\hline
 & \textbf{K3 surface} & \textbf{Group} & $\textbf{Singularities}$ \\
\hline
(1) & $X_{4,4} \subset \mathbb{P}(1, 1, 2, 2, 2)$ & $\mathbb{Z}/4\times(\mathbb{Z}/2)^3$ & $4A_1$ \\ \hline
(2) & $X_{8} \subset \mathbb{P}(1, 1, 2, 4)$ & $\mathbb{Z}/8\times \mathbb{Z}/4\times\mathbb{Z}/2$ & $2A_1$ \\ \hline
(3) & $X_{6} \subset \mathbb{P}(1,1,2,2)$ & $\mathbb{Z}/6\times (\mathbb{Z}/3)^2$ & $3A_1$ \\ \hline
(4) & $X_{12} \subset \mathbb{P}(1, 3, 4, 4)$ & $\mathbb{Z}/4\times (\mathbb{Z}/3)^2$ & $3A_3$ \\ \hline
(5) & $X_{12} \subset \mathbb{P}(1, 2, 3, 6)$ & $\mathbb{Z}/6\times \mathbb{Z}/4\times\mathbb{Z}/2$ & $2A_1 + 2A_2$ \\ \hline
(6) & $X_{12} \subset \mathbb{P}(2, 3, 3, 4)$ & $(\mathbb{Z}/4)^2\times \mathbb{Z}/3$ & $3A_1 + 4A_2$ \\ \hline
(7) & $X_{6,6} \subset \mathbb{P}(1, 2, 3, 3, 3)$ & $\mathbb{Z}/3\times(\mathbb{Z}/2)^3$ & $4A_2$ \\ \hline
(8) & $X_{6,6} \subset \mathbb{P}(2, 2, 2, 3, 3)$ & $(\mathbb{Z}/3)^2\times\mathbb{Z}/2$ & $9A_1$ \\ \hline
\end{longtable}
\emph{Table 4. Examples of singular K3 surfaces with the action of a finite abelian group}
\end{center}
\endgroup
\end{example}

Recall that a Kummer K3 surface is the minimal resolution of the quotient of an abelian surface $A$ by the multiplication by $-1$. 
In the next example, we show the existence of a Kummer K3 surface with the action of the group $H=\mathbb{Z}/4\times (\mathbb{Z}/2)^3$. 

\begin{example}
\label{ex-kummer-1}
Let $C_1$ be an elliptic curve with a faithful action of $H'=(\mathbb{Z}/2)^3$, and $C_2$ be an elliptic curve with a faithful action of $H''=\mathbb{Z}/2\times\mathbb{Z}/4$ (so that the $j$-invariant of $C_2$ is equal to $1728$). Note that $H'$ acts on the set of $2$-torsion points $\{a_1,a_2,a_3,a_4\}$ of $C_1$ transitively, while $H''$ acts on the set of $2$-torsion points $\{b_1,b_2,b_3,b_4\}$ of $C_2$ with two orbits of cardinality $2$, say, it interchanges $b_1$ with $b_3$ and $b_2$ with $b_4$. 
Consider an abelian surface $A=C_1\times C_2$, which admits the action of $H'\times H''=(\mathbb{Z}/2)^4\times \mathbb{Z}/4$. It follows that $H'\times H''$ has $2$ orbits of cardinality $8$. 
Consider the quotient $S=A/\sigma$ where $\sigma=(\sigma_1,\sigma_2)$, and $\sigma_i$ is the multiplication by $-1$ on $C_i$. Then $S$ is a K3 surface with $16$ singular points of type $A_1$. From the construction it follows that $S$ admits a faithful action of $H=(H'\times H'')/(\mathbb{Z}/2)=(\mathbb{Z}/2)^3\times \mathbb{Z}/4$. One checks that if $C_1$ and $C_2$ are isogenous then $\rho(S')=20$ where $S'$ is the minimal resolution of $S$.
\end{example}

\begin{example}
Similarly to Example \ref{ex-kummer-1}, one can construct a Kummer K3 surface $S$ 
with $16$ $A_1$ singularities 
that admits a faithful action of $H=(\mathbb{Z}/2)^5$. Note that all the singular points of $S$ lie in the same $H$-orbit in this case. As in the previous case, the Picard rank of the minimal resolution $S'$ of $S$ is equal to $20$.
\end{example}

\section{K3 surfaces: lattices}
\label{sec-K3-lattices}
Let $S$ be a smooth projective K3 surface. 
Consider the second cohomology group $\mathrm{H}^2(S, \mathbb{Z})$ as a lattice $\Lambda_{K3}$ endowed with the cup product. It is well-known that $\Lambda_{K3}=U^{\oplus 3}\oplus E_8(-1)^{\oplus 2}$ where $U$ is the hyperbolic lattice, $E_8$ is the unique positive-definite, even, unimodular lattice of rank $8$, and $E_8(-1)$ means that we multiply the bilinear form of $E_8$ by $-1$. 
Recall that on a K3 surface~$S$ we have $\mathrm{Pic}(S)=\mathrm{NS}(S)$, and 
\[
\mathrm{NS}(S) = \{ x\in \mathrm{H}^2(S, \mathbb{Z})\ |\ x\cdot \omega_S = 0\}=\mathrm{H}^{1,1}(S, \mathbb{Z}),
\]
where $\omega_S$ is (the class of) a non-zero holomorphic $2$-form on $S$. 
The transcendental lattice $T(S)$ is defined as follows:
\[
T_S = \mathrm{NS}(S)^\perp \subset \mathrm{H}^2(S, \mathbb{Z}).
\]
It follows that the sublattice
\[
\mathrm{NS}(S) \oplus T_S \subset \mathrm{H}^2(S, \mathbb{Z})
\]
has finite index, equal to the discriminant of $\mathrm{NS}(S)$ and $T_S$. 

\begin{example}
\label{exam-fermat-quartic}
For a Fermat quartic surface $S_4\subset \mathbb{P}^3$, according to \cite[3.2.6]{Hu16}, one has
\[
\mathrm{NS}(S_4) = E_8(-1)^{\oplus2} \oplus U \oplus \mathbb{Z}(-8) \oplus \mathbb{Z}(-8), \quad \quad \quad T_S = \mathbb{Z}(8) \oplus \mathbb{Z}(8).
\]
In particular, one has $\rho(S_4)=20$.
\end{example}

\begin{lem}
Let $S$ be a smooth K3 surface with a faithful action of a finite group $H$. 
If the action of $H$ on $S$ is non-symplectic, then $\mathrm{NS}(S)^H=\mathrm{H}^2(S,\mathbb{Z})^H$.
\end{lem}
\begin{proof}
Consider an element $x\in \mathrm{H}^2(S,\mathbb{Z})^H$. Consider its decomposition $x=x_{2,0}+x_{1,1}+x_{0,2}$ as an element of $\mathrm{H}^2(S,\mathbb{C})$ where $x_{2,0}\in \mathrm{H}^{2,0}(S)$, $x_{1,1}\in \mathrm{H}^{1,1}(S)$, and $x_{0,2}\in \mathrm{H}^{0,2}(S)$. Then for a non-symplectic element $h\in H$ we have 
\[
h^*(x) = h^*(x_{2,0})+h^*(x_{1,1})+h^*(x_{0,2})=x_{2,0} + x_{1,1} + x_{0,2}=x\textbf{}.
\]
Since $h$ acts on the generator of $\mathrm{H}^{2,0}(S)$ (respectively, of $\mathrm{H}^{0,2}(S)$) non-trivially, it follows that $x_{2,0}=0$ (respectively, $x_{0,2}=0$). Hence $x=x_{1,1}\in \mathrm{H}^2(S,\mathbb{Z})\cap \mathrm{H}^{1,1}(S)=\mathrm{NS}(S)$. Thus, $\mathrm{H}^2(S,\mathbb{Z})^H\subset \mathrm{NS}(S)^H$, and the claim follows.
\end{proof}

Using Corollary \ref{cor-max-group-direct-product}, we obtain the following. 

\begin{corollary}
For the groups (1)--(6) as in Theorem \ref{thm-maximal-k3-groups} we have $\mathrm{NS}(S)^H=\mathrm{H}^2(S,\mathbb{Z})^H$.     
\end{corollary}

The following proposition is crucial to our work. 

\begin{proposition}
\label{prop: determinants}
\label{prop-ranks-H0}
	Let $H$ be one of the groups (1)--(6) in Theorem \ref{thm-maximal-k3-groups}, and let $S$ be a smooth K3 surface with a faithful action of $H$. Let $M$ be the intersection matrix on $\mathrm{H}^2(S,\mathbb{Z})^H$. 
    Put $r=\mathrm{rk}\,\mathrm{H}^2(S, \mathbb{Z})^H$. 
    Then one of the following cases holds.
	\begin{enumerate}
        \item $H=(\dz/4)^3,\ r=1,\ M=(4)$,
        \item $H=(\dz/6)^2\times\dz/2,\ r=1,\ M=(2)$,
        \item $H=\dz/6\times(\dz/3)^2,\ r=2,\ M=\begin{pmatrix}0&3\\3&0\end{pmatrix}$,
        \item $H=\dz/8\times\dz/4\times\dz/2,\ r=2,\ M=\begin{pmatrix}0&2\\2&0\end{pmatrix}$,
		\item $H=(\dz/2)^5,\ 1\leq r\leq5$,
		\item $H=\dz/4\times(\dz/2)^3,\ 2\leq r\leq 6$.
	\end{enumerate}
\end{proposition}
\begin{proof}
Follows from the database provided by \cite{BH23}.
\end{proof}

The possible intersection matrices for the cases (5) and (6) in Proposition \ref{prop: determinants} are presented in the Appendix.

\begin{lem}
\label{lem-r-1-2}
Assume that a finite group $H$ faithfully acts on a K3 surface $S$ with at worst du Val singularities. Let $f\colon S'\to S$ be the minimal resolution, which is automatically $H$-equivariant. Put $r=\mathrm{rk}\,\mathrm{H}^2(S', \mathbb{Z})^H$. Then the following holds. 
\begin{enumerate}
\item If $r=1$, then $S$ is smooth. 
\item 
If $r=2$, then the $f$-exceptional $(-2)$-curves on $S'$ form one $H$-orbit, and the singular points of $S$ form one $H$-orbit.
\end{enumerate}
\end{lem} 
\begin{proof}
Assume that $r=1$. 
Then for the $H$-invariant Picard number of $S'$ we have $\rho^H(S')=1$. 
Since $S$ is projective, we get $\rho^H(S)\geq 1$, and so $\rho^H(S)=1$. It follows that~$f$ is an isomorphism, so $S$ is smooth. 

Assume that $r=2$. Since the $f$-exceptional $(-2)$-curves are linearly independent in $\mathrm{NS}(S')$, it follows that the $H$-orbits of $f$-exceptional $(-2)$-curves are linearly independent in $\mathrm{NS}^H(S')$. Let $A$ be an $H$-invariant ample divisor on $S$, and let $A'=f^*A$ be an $H$-invariant divisor on $S'$ with $A'^2>0$. Note that the $f$-exceptional $(-2)$-curves are orthogonal to $A'$. Since $r=2$, we conclude that $\rho^H(S')\leq 2$. Thus, the $f$-exceptional $(-2)$-curves on $S'$ form one $H$-orbit, and hence the singular points on $S$ form one $H$-orbit.
\end{proof}

\begin{corollary}
\label{cor-rank-1-is-smooth}
    Let $H$ be a group isomorphic to $(\dz/4)^3$ or $(\dz/6)^2\times\dz/2$. Then any K3 surface with a faithful action of $H$ is smooth.
\end{corollary}
\begin{proof}
Follows from Proposition \ref{prop-ranks-H0} and Lemma \ref{lem-r-1-2}.
\end{proof}

\begin{corollary}
\label{cor-rank-2-is-transitive}
    Let $H$ be a group isomorphic to $\dz/8\times\dz/4\times\dz/2$ or $\dz/6\times(\dz/3)^2$ faithfully acting on a K3 surface $S$. Let $f\colon S'\to S$ be the minimal resolution. 
    Then $H$ acts on the set of $f$-exceptional $(-2)$-curves transitively. In particular, singularities on $S$ form one $H$-orbit, and they can be only of type $A_1$ or $A_2$.  
\end{corollary}
\begin{proof}
Follows from Proposition \ref{prop-ranks-H0} and Lemma \ref{lem-r-1-2}. The last claim follows by looking at the dual graphs of the minimal resolution of du Val singularities. 
\end{proof}

\begin{remark}
Let $S$ be a K3 surface with du Val singularities endowed with the action of a finite abelian group $H$. 
Let $f\colon S'\to S$ be the minimal resolution. Note that $f$ is automatically $H$-equivariant. Then 
\[
\mathrm{NS}(S')_{\mathbb{Q}}=f^*\mathrm{NS}(S)_{\mathbb{Q}}\oplus V_{\mathbb{Q}}
\]
where $V=\langle E_i\rangle$ is the subgroup in $\mathrm{NS}(S')$ spanned by the $f$-exceptional $(-2)$-curves $E_i$, and $V_{\mathbb{Q}}=V\otimes \mathbb{Q}$. However, it is not true that 
\[
\mathrm{NS}(S')=f^*\mathrm{NS}(S)\oplus V.
\]
Indeed, otherwise we would have 
\[
\mathrm{NS}(S')^H=f^*\mathrm{NS}(S)^H\oplus V^H,
\]
which is not the case, as the following example shows. 
\end{remark}

\begin{example}
Consider a K3 surface $S$ as in Example \ref{exam-singular-k3}(2) with the action of the group $H=\mathbb{Z}/8\times\mathbb{Z}/4\times \mathbb{Z}/2$. Let $f\colon S'\to S$ be the minimal resolution. By Proposition~\ref{prop: determinants} we have that $\Lambda=\mathrm{NS}(S')^H$ is the lattice with the Gram matrix 
\[
B_\Lambda = \begin{pmatrix}0&2\\2&0\end{pmatrix}.
\]
Put 
\[
\Lambda'=f^*\mathrm{NS}(S)^H\oplus V^H
\]
where $V=\langle E_i\rangle$ is a subgroup in $\mathrm{NS}(S')$ spanned by the $f$-exceptional $(-2)$-curves $E_1$ and $E_2$, and $V^H$ is spanned by their $H$-orbit $E_1+E_2$. Then $\Lambda'$ has the Gram matrix 
\[
B_{\Lambda'} = \begin{pmatrix}4&0\\0&-4\end{pmatrix}. 
\]
We have that $\rho^H(S)=1$, and $\mathrm{NS}(S)^H=\mathbb{Z}[2A]$ where $A$ is the restriction of $\OOO(1)$ in $\mathbb{P}(1,1,2,4)$ to $S$, so that $A^2=1$. Then the lattice $2\Lambda = 2\mathrm{NS}(S')^H$, which has the Gram matrix 
\[
B_{2\Lambda} = \begin{pmatrix}0&8\\8&0\end{pmatrix},
\] 
is a sublattice of $\Lambda'$. We have the inclusions $2\Lambda \subset \Lambda'\subset \Lambda$. 
\end{example}

\begin{proposition}
\label{prop-inclusion-of-lattices}
We have
\[
\mu\mathrm{NS}(S')\subset f^*\mathrm{NS}(S)\oplus V\subset \mathrm{NS}(S').
\]
where $V=\langle E_i\rangle$ is a subgroup in $\mathrm{NS}(S')$ spanned by the $f$-exceptional $(-2)$-curves $E_i$, and $\mu$ is the index of $\mathrm{Cl}(S)$ in $\mathrm{NS}(S)=\mathrm{Pic}(S)$. Similarly, we have
\begin{equation}
\label{eq-inclusion-of-lattices}
\mu\mathrm{NS}(S')^H\subset f^*\mathrm{NS}(S)^H\oplus V^H\subset \mathrm{NS}(S')^H.
\end{equation}
Consequently, 
\begin{equation}
\label{eq-determinant-of-lattices}
\mathrm{disc} (\mathrm{NS}(S')^H)\quad |\quad \mathrm{disc} (\mathrm{NS}(S)^H)\det V\quad |\quad \mu^{2\rho} \mathrm{disc} (\mathrm{NS}(S')^H),
\end{equation}
where $\rho$ is the rank of $\mathrm{NS}(S')^H$.
\end{proposition}
\begin{proof}
Let $D\in \mathrm{NS}(S')^H$. Then $f^*f_*D=D+\sum a_iE_i$ where $E_i\in V$ and $a_i\in\frac{1}{\mu}\mathbb{Z}$. Hence, $\mu f_*D\in\mathrm{NS}(S)^H$, and $\mu f^*f_*D$ belongs to $f^*\mathrm{NS}(S)^H$. Thus, $\mu (f^*f_*D-D)\in V$. This shows \eqref{eq-inclusion-of-lattices}. Then \eqref{eq-determinant-of-lattices} follows by taking the discriminants of the lattices in \eqref{eq-inclusion-of-lattices}.
\end{proof}

\section{Terminal singularities}
\label{sec-terminal-sing-classification}
In this section we recall the classification of three-dimensional terminal singularities, cf. \cite{Mo85}, \cite{Re85}. 
Let $x\in X$ be a germ of a three-dimensional terminal singularity. Then the singularity is isolated: $\mathrm{Sing}(X) = \{x\}$. The \emph{index} of $x\in X$ is the minimal positive integer $r$ such that $rK_X$ is Cartier. 

If $r = 1$, then $x\in X$ is Gorenstein. In this case $x\in X$ is analytically isomorphic to a hypersurface singularity in $\mathbb{C}^4$ of multiplicity~$2$. Moreover, any Weil $\mathbb{Q}$-Cartier divisor $D$ on $x\in X$ is Cartier. Also, in this case $x\in X$ is a compound du Val singularity, which means that its general hyperplane section $H$ that contains $x$ is a surface with a du Val singularity at $x$. We say that $x\in X$ has type $cA$, $cD$, $cE$, respectively, if $x\in H$ is a du Val singularity of type $A$, $D$, $E$, respectively. 

If $r > 1$, then there is a cyclic \'etale outside~$x$ covering 
\[
\pi\colon \widetilde{x}\in \widetilde{X} \to X\ni x
\] 
of degree $r$ such that $\widetilde{x}\in \widetilde{X}$ is a Gorenstein terminal singularity (or a smooth point), and $\pi^{-1}(x)=\widetilde{x}$. The map $\pi$ is called the \emph{index-one cover} of $x\in X$, and it is defined canonically.

\begin{thm}[{\cite{Mo85, Re85}}]
Let $x\in X$ be a three-dimensional terminal singularity of index $r > 1$. Then $x\in X$ is analytically isomorphic to the quotient $\{\phi = 0\}/(\mathbb{Z}/r)$ of a hypersurface $\mathbb{C}^4$ defined by the equation
\[
\phi(x_1, \dots, x_4) = 0,
\]
where the group $\mathbb{Z}/r$ acts on $\mathbb{C}^4$ such that the coordinates $x_i$ and the equation $\phi$ are semi-invariant. Moreover, up to an analytic $\mathbb{Z}/r$-equivariant coordinate change, the hypersurface $\phi = 0$ and the $\mathbb{Z}/r$-action are described by one of the rows in the following table (where $\mathfrak{m}$ denotes the maximal ideal of $x\in X$, and the last number in the weights vector is the weight of $\phi$): 
\end{thm}

\ 

\begingroup
\renewcommand{\arraystretch}{1.2}
\begin{longtable}{|p{0.11\textwidth}|p{0.10\textwidth}|p{0.45\textwidth}|p{0.20\textwidth}|}
\hline
\textbf{Type} & \textbf{Index} & \textbf{$\phi(x_1,x_2,x_3,x_4)$} & \textbf{Weights} \\ \hline
$cA/r$ & $r\geq 1$ & 
$x_1x_2 + \psi(x_3^r, x_4)$ & 
$(1,-1,a,0;0)$, $(r,a)=1$  \\ \hline
$cAx/2$ & $r=2$ & 
$x_1^2 + x_2^2 + \psi(x_3, x_4)$, $\psi\in\mathfrak{m}^4$ & 
$(0,1,1,1;0)$\\ \hline
$cD/2$ & $r=2$ & 
$x_4^2 + \psi(x_1,x_2,x_3)$, $\psi \in \mathfrak{m}^3$ 
with $x_1x_2x_3$ or $x_2^2x_3 \in \psi$ & 
$(1,1,0,1;0)$ \\ \hline
$cE/2$ & $r=2$ & 
$x_4^2 + x_1^3 + \psi(x_2,x_3)x_1 + \theta(x_2,x_3)$, $\theta \notin \mathfrak{m}^5$ & 
$(0,1,1,1;0)$ \\ \hline
$cD/3$ & $r=3$ & 
$x_4^2 + \psi(x_1,x_2,x_3)$ 
with $\psi_3 = x_1^3 + x_2^3 + x_3^3$ or  
$x_1^3 + x_2x_3^2$ or $x_1^3 + x_2^3$ & 
$(1,2,2,0;0)$ \\ \hline
$cAx/4$ & $r=4$ & 
$x_1^2 + x_2^2 + \psi(x_3^2, x_4)$, $\psi \in \mathfrak{m}^3$ & 
$(1,3,1,2;2)$ \\ \hline
\end{longtable}
\begin{center}
\emph{Table 5. Terminal singularities} 
\end{center}
\endgroup

\subsection{Basket of singularities}
\label{sec-about-baskets}
For a three-dimensonal terminal singularity $x \in X$ there is a deformation to $k \geq 1$ terminal cyclic quotient singularities $x_1, \dots, x_k$. 
We may
assume that the singularities $x_i$ have type $\frac{1}{r_i}(1, -1, b_i)$ where $0 < b_i \leq r_i/2$. The set $\{x_1, \dots, x_k\}$ is called the \textit{basket of singularities}
of $x\in X$ and it can be written as
\[
B(x \in X) = \left\{ n_i \times \frac{1}{r_i}(1, -1, b_i)\right\}.
\]
By the basket of singularities of $X$, denoted by $B(X)$, we mean the union of all baskets of $x\in X$ for all non-Gorenstein singular points $x\in X$. 

If $x\in X$ is a three-dimensonal non-Gorenstein terminal singularity of index $r$, then in its basket of cyclic quotient singularities all the points in the basket have index $r$, except in the case when $x\in X$ is of type $cAx/4$, in which case one of the points in the basket has index $4$, and all the other points in the basket have index $2$. 
Moreover, if $x\in X$ is not a quotient singularity itself, then in the basket of $x\in X$ there are at least two points. 

\subsection{Orbifold Riemann-Roch}
By \cite[10.2]{Re85}, for a projective terminal threefold $X$ and a Weil $\mathbb{Q}$-Cartier divisor $D$ on it we have the following version of the Riemann-Roch formula:
\begin{equation}
\label{eq-ORR}
\chi (\OOO_X(D)) = \chi (\OOO_X) + \frac{1}{12}D(D-K_X)(2D-K_X) + \frac{1}{12}D\cdot c_2(X) + \sum_Q c_Q (D)
\end{equation}
where for any cyclic quotient singularity $Q$ we have 
\begin{equation}
c_Q(D) = -i \frac{r^2-1}{12r} + \sum_{j=1}^{i-1}\frac{\overline{bj} (r-\overline{bj}) }{2r}.
\end{equation}
Here $r$ is the index of $Q$, the divisor $D$ has type $i \frac{1}{r}(a, -a, 1)$ at $Q$, $b$ satisfies $ab=1 \ \mathrm{mod}\ r$, and $\overline{\ }$ denotes the residue modulo $r$. For non-cyclic non-Gorenstein singularities, their contribution to the right-hand side of \eqref{eq-ORR} is computed in terms of their basket of cyclic points. Moreover, by \cite[10.3]{Re85} (see \cite[Theorem 12.1.3]{Pr21} in the case of $G\mathbb{Q}$-Fano threefolds) one has 
\begin{equation}
\label{eq-Miyaoka}
(-K_X)\cdot c_2(X) + \sum_Q (r-1/r) = 24
\end{equation}
Since $(-K_X)\cdot c_2(X)>0$, we see that the number of non-Gorenstein singular points is at most $15$.

\subsection{Geometry of the flag $x\in S\subset X$}
We start with the following lemma which is well-known to experts. 

\begin{lem}
\label{lem-singular-du-val}
Let $x\in X$ be a germ of a threefold terminal singularity. Let $S\in |-K_X|$ be an anti-canonical element, and assume that $x\in X$ is not a smooth point. Then $x$ is a singular point on $S$. 
\end{lem}
\begin{proof}
The claim is clear if $x$ is a Gorenstein point. Assume that $x$ is non-Gorenstein of index $r>1$. Consider the index-one cover $\pi\colon \widetilde{X}\to X$ which is \'etale of degree $r$ outside~$x$. Assume that $S$ is smooth at $x$. Put $\pi^{-1}(S)=\widetilde{S}$. Then $\pi|_{\widetilde{S}\setminus \widetilde{x}}\colon \widetilde{S}\setminus \{\widetilde{x}\}\to S\setminus \{x\}$ is \'etale outside $x$ as well. Since by assumption $S$ is smooth at $x$, it follows that $S\setminus \{x\}$ is simply-connected. Hence the cover $\pi|_{\widetilde{S}\setminus \{\widetilde{x}\}}\colon \widetilde{S}\setminus \{\widetilde{x}\}\to S\setminus \{x\}$ splits, so that $\widetilde{S}=\sum \widetilde{S}_i$, and $\widetilde{S}_i\cap \widetilde{S}_j=\{\widetilde{x}\}$ for $i\neq j$. However, since $rS$ is Cartier, it follows that $r\widetilde{S}$ is Cartier as well, and hence Cohen-Macaulay (here we use the fact that three-dimensional terminal singularities are Cohen-Macaulay). Thus, $\widetilde{S}\setminus \{\widetilde{x}\}$ should be connected, which is not the case. 
This leads to a contradiction, which shows that $S$ is singular at $x$. The claim follows.
\end{proof}


\begin{remark}
\label{rem-index-divides-exponent}
Let $x\in S$ be a germ of a du Val singularity. 
Let $\pi_1^{ab}(S\setminus \{x\})$ be the abelianization of the local fundamental group. Then there are the following possibilities for $\pi_1^{ab}(S\setminus \{x\})$ according to the type of singularity:
\begin{enumerate}
\item
$\mathbb{Z}/(n+1)$ for type $A_n$, 
\item
$\mathbb{Z}/2\oplus\mathbb{Z}/2$ for type $D_n$ for even $n\geq 4$, 
\item
$\mathbb{Z}/2$ for type $D_n$ for odd $n\geq 5$, 
\item
$\mathbb{Z}/3$ for type $E_6$, 
\item
$\mathbb{Z}/2$ for type $E_7$, 
\item
$0$ for type $E_8$. 
\end{enumerate}
It is also known that the exponent of $\pi_1^{ab}(S\setminus \{x\})$ is equal to the index of the local class group $\mathrm{Cl}(x\in S)$ in the local Picard group $\mathrm{Pic}(x\in S)$.  
\end{remark}

Let $x\in X$ be the germ of a threefold terminal singularity of index $r\geq 1$, and let $S\in|-K_X|$ be an anti-canonical element. Assume that the pair $(X,S)$ is plt. 
Let $\pi\colon \widetilde{x}\in \widetilde{X} \to X\ni x$ be the index-one cover. 
Note that $\pi$ induces a cover $\pi|_{\widetilde{S}}\colon \widetilde{x}\in \widetilde{S}\to S\ni x$. 
There exists the following diagram:
\begin{equation}
\label{diagram-terminal-flag}
\begin{tikzcd}
\widetilde{x}\in \widetilde{X} \ar[r, "\pi"] & X\ni x \\
\widetilde{x}\in \widetilde{S} \ar[u, hook] \ar[r, "\pi|_{\widetilde{S}}"] & S\ni x  \ar[u, hook]
\end{tikzcd}
\end{equation}

\begin{proposition}
Assume that 
the induced map $\pi|_{\widetilde{S} \setminus \{\widetilde{x}\}}\colon\widetilde{S}\setminus \{\widetilde{x}\}\to S\setminus \{x\}$ is a 
non-split cyclic \'etale covering of degree $k$. Then $r$ divides $|\pi_1^{ab}(S\setminus \{x\})|$.
\end{proposition}
\begin{proof}
Note that $\pi_{\widetilde{S}\setminus \{\widetilde{x}\}}$ corresponds to a subgroup $H$ of $\pi_1^{\mathrm{ab}}(S\setminus \{x\})$ such that $\pi_1^{\mathrm{ab}}(S\setminus \{x\})/H=\mathbb{Z}/k$. In particular, if $x\in S\subset X$ where $x\in X$ is a non-Gorenstein threefold terminal point of index $r$ and $S\in |-K_X|$ has du Val singularities, then $\pi|_{\widetilde{S} \setminus \{\widetilde{x}\}}$ is non-split (see the proof of Lemma~\ref{lem-singular-du-val}) cyclic \'etale covering of degree~$r$, hence $r$ divides $|\pi_1^{ab}(S\setminus \{x\})|$. 
\end{proof}

We consider two cases: when $x\in X$ is a cyclic quotient singularity, so $(\widetilde{x}\in \widetilde{X})\simeq (0\in \mathbb{C}^3)$, and when 
$x\in X$ is not a cyclic quotient singularity, so $(\widetilde{x}\in \widetilde{X})\subset (0\in \mathbb{C}^4)$. Note that in the latter case $\widetilde{S}$ is a Cartier divisor on $\widetilde{X}\subset \mathbb{C}^4$, hence $\widetilde{S}$ is singular at $\widetilde{x}$. Moreover, since by assumption the pair $(X, S)$ is plt, it follows that the pair $(\widetilde{X}, \widetilde{S})$ is plt as well, so $\widetilde{x}\in \widetilde{S}$ is a du Val singularity. 


\begin{proposition}
\label{prop-s-tilde-is-smooth}
Assume that $x\in X$ is a threefold cyclic quotient singularity of index $r\geq 1$. 
If $\widetilde{x}\in\widetilde{S}$ is smooth then $x\in S$ has type $A_{r-1}$.
\end{proposition}
\begin{proof}
By assumption, $\widetilde{x}\in \widetilde{S}$ is smooth, so $(\widetilde{x}\in \widetilde{S}) \simeq (0\in\mathbb{C}^2)$. Hence the non-split degree~$r$ cyclic covering $\pi|_{\widetilde{S}\setminus \{\widetilde{x}\}}\colon \widetilde{S}\setminus \{\widetilde{x}\}\to S\setminus \{x\}$ is a universal covering. 
It follows that $\pi_1(S\setminus\{x\})$ is a group of order~$r$. 
Since $\pi|_{\widetilde{X}\setminus \{\widetilde{x}\}}$ is a cyclic covering, the group of deck transformations of $\pi|_{\widetilde{X}\setminus \{\widetilde{x}\}}$ is cyclic of order~$r$. 
Thus, the group of deck transformations of $\pi|_{\widetilde{S}\setminus \{\widetilde{x}\}}$ contains an element of order~$r$. Since $|\pi_1(S\setminus\{x\})|=r$, it follows that $\pi_1(S\setminus\{x\})=\mathbb{Z}/r$. 
This implies that $x\in S$ is a singular point of type $A_{r-1}$.
\end{proof}

\subsection{Group action on a terminal singularity}
\label{subsec-action-on-terminal}
Let $X$ be a $G\mathbb{Q}$-Fano threefold where $G$ is a finite abelian group. Let $x\in X$ be a germ of a terminal singularity. Assume there exists a $G$-invariant element $S\in |-K_X|$ which is a K3 surface with du Val singularities. 
Then there is an exact sequence
\begin{equation}
\label{exact-sequence-m-G-H}
0\to \mathbb{Z}/m\to G\to H \to 0,    
\end{equation}
where $H$ faithfully acts on $S$, and $\mathbb{Z}/m$ faithfully acts in the normal bundle to $S$ in $X$ for some $m\geq 1$. 
Let $G_x$ and $H_x$ be the stabilizers of $x$ in $G$ and $H$, respectively. Thus, we obtain the exact sequence
\begin{equation}
\label{exact-sequence-m-G-H-x}
0\to \mathbb{Z}/m\to G_x\to H_x \to 0.    
\end{equation}

\begin{thm}[{\cite[Theorem 7.3]{Lo24}}]
\label{thm-action-on-terminal-point}
Let $x\in X$ be a germ of a threefold terminal singularity and let
$G_x \subset \mathrm{Aut}(x\in X)$ be a finite abelian subgroup. Then either $\mathfrak{r}(G_x)\leq 3$, or 
\[
G_x = (\mathbb{Z}/2)^2\times\mathbb{Z}/2n\times\mathbb{Z}/2m
\]
for $n,m\geq 1$. 
Moreover, in the latter case $x\in X$ is a Gorenstein singularity of type $cA$. In particular, in both cases $G$ is of product type.
\end{thm}

The diagram \eqref{diagram-terminal-flag} induces the following diagram:
\begin{equation}
\begin{tikzcd}
\label{diagram-two-levels-of-groups}
0 \ar[r] & \mathbb{Z}/r \ar[r] & \widetilde{G}_x \ar[d] \ar[r] & G_x \ar[d] \ar[r] & 0 \\
0 \ar[r] & \mathbb{Z}/r \ar[u, equal] \ar[r]  & \widetilde{H}_x  \ar[r] & H_x  \ar[r] & 0
\end{tikzcd}
\end{equation}
where $\widetilde{G}_x$ is the lifting of $G_x$, and $\widetilde{H}_x$ is the lifting of $H_x$. This means that $\widetilde{G}_x$ faithfully acts on $\widetilde{x}\in \widetilde{X}$, and $\widetilde{H_x}$ faithfully acts on $\widetilde{x}\in\widetilde{S}$. 

\begin{proposition}
\label{prop-s-tilde-is-singular}
Assume that $x\in X$ is a threefold cyclic quotient singularity of index $r\geq 1$. 
If $\widetilde{x}\in\widetilde{S}$ is singular then $G$ is of product type.  
\end{proposition}
\begin{proof}
Since $T_{\tilde{x}}\widetilde{S} = T_{\tilde{x}}\widetilde{X}=\mathbb{C}^3$, by Lemma \ref{lem-faithful-action} we know that $\widetilde{G}_x=\widetilde{H}_x$, and hence $G_x=H_x$. It follows that in the exact sequences \eqref{exact-sequence-m-G-H-x} and \eqref{exact-sequence-m-G-H} we have $m=1$. Thus, we have $G=H$. Therefore $G$ is of product type by Remark \ref{rem-special-K3-groups-are-PT}. 
\end{proof}

\begin{corollary}
\label{cor-s-s-tilde}
Assume that $x\in X$ is a threefold cyclic quotient singularity of index $r\geq 1$. Then either $G$ is of product type, or $\widetilde{x}\in\widetilde{S}$ is smooth and $x\in S$ is a singularity of type $A_{r-1}$.
\end{corollary}
\begin{proof}
Follows from Proposition \ref{prop-s-tilde-is-smooth} and Proposition \ref{prop-s-tilde-is-singular}.
\end{proof}

\begin{proposition}
\label{prop-lift-of-2-group-is-abelian}
Assume that 
\begin{itemize}
    \item either $r>2$ ,
    \item or $G_x$ is a $2$-group.
\end{itemize}
Then the lifting $\widetilde{G}_x$ is abelian.
\end{proposition}
\begin{proof}
If the first assumption of the proposition holds, then the claim follows by \cite[Proposition 7.13]{Lo24}. 
If the first assumption of the proposition holds, then the claim follows by \cite[Proposition 7.18]{Lo24}.
\end{proof}

\begin{remark}
\label{rem-action-in-normal-bundle-splits}
Assume that 
$x\in X$ is a threefold cyclic quotient singularity of index $r\geq 1$. 
Assume that $\widetilde{x}\in \widetilde{S}$ is smooth, so $(\widetilde{x}\in \widetilde{S}) \simeq (0\in \mathbb{C}^2)$.  
Considering $\mathbb{C}^2$ as a linear subspace in $\mathbb{C}^3$, we see that there exists an injective homomorphism $\widetilde{H}_x\to \widetilde{G}_x$ which is inverse to the map $\widetilde{G}_x\to \widetilde{H}_x$ in~\eqref{diagram-two-levels-of-groups}. 
In particular, we have $\widetilde{G}_x=\widetilde{H}_x\times \mathbb{Z}/m$.
\end{remark}

\begin{proposition}
\label{prop-non-gorenstein-exclusion}
Assume that $x\in X$ is a threefold terminal non-Gorenstein singularity which is not a cyclic quotient singularity. 
Then the singular point $x\in S$ cannot be of type $A_1$ or $A_2$. In particular, the action of $H$ on the set of $(-2)$-curves on the minimal resolution $S'$ of $S$ cannot be transitive. 
\end{proposition}
\begin{proof}
By Lemma \ref{lem-singular-du-val}, the point $x\in S$ is singular. 
Since $x\in X$ is not a cyclic quotient singularity, we have that $\widetilde{x}\in \widetilde{X}$ is singular. Since $\widetilde{S}$ is Cartier at the point $\widetilde{x}$ in $\widetilde{X}$, we see that the point $\widetilde{x}\in \widetilde{S}$ is singular as well, and it is a du Val singularity. Hence the map 
$\pi|_{\widetilde{S}\setminus\{\widetilde{x}\}}\colon \widetilde{S}\setminus \{\widetilde{x}\}  \to  S\setminus\{ x\}$ is a non-trivial covering which is not universal. Thus, $\widetilde{x}\in \widetilde{S}$ cannot be of type $A_1$ or $A_2$, because in these cases $S\setminus\{ x\}$ does not admit such a covering, cf. Remark \ref{rem-index-divides-exponent}. The last claim follows from looking at the dual graphs of the exceptional curves on the minimal resolution of $S$. 
\end{proof}

\section{Case $h^0(-K_X)\geq 2$}
\label{sec-h0-geq-2}
Let $X$ be a $G\mathbb{Q}$-Fano threefold where $G$ is a finite abelian group. 
Throughout this section we assume that $h^0(X, \OOO(-K_X))\geq 2$, and that all the $G$-invariant elements in $|-K_X|$ are K3 surfaces with at worst canonical singularities. Let $S$ and $S'$ be two such surfaces. Consider exact sequences
\begin{equation}
\label{exact-sequence-restriction-to-K3}
0\to C \to G \to H \to 0, \quad \quad \quad 
0\to C' \to G \to H' \to 0
\end{equation}
where $H$ (respectively, $H'$) faithfully acts on $S$ (respectively, $S'$), and $C=\mathbb{Z}/m$ (respectively, $C'=\mathbb{Z}/m'$) fixes $S$ (respectively, $S'$) pointwise. 
\begin{lem}
\label{lem-C1C2G1G2}
We have that $C$ (respectively, $C'$) faithfully acts on $S'$ (respectively, $S$), and this action is purely non-symplectic. In particular, we have $C\cap C'=\{\id\}$, and the maps $C\hookrightarrow H'$ and  $C'\hookrightarrow H$ induced by the exact sequences \eqref{exact-sequence-restriction-to-K3}  are injective. 
\end{lem}
\begin{proof}
Assume that $C$ does not act faithfully on $S'$. Then there is a non-trivial element $g\in C$ acting trivially on $S'$. The fixed locus of $g$ contains $S\cup S'$, hence it is singular along a curve $S\cap S'$. This is impossible by Lemma \ref{cor-fixed-curve-surface}.  
Finally, expressing the volume form on $S'$ in local coordinates at the general point $x\in S\cap S'$, we see that the action of $C$ on $S'$ is purely non-symplectic. The claim for the action of $C'$ on $S$ follows by symmetry.
\end{proof}

From Remark \ref{rem-k3-or-product}, Corollary \ref{cor-max-group-direct-product} and Lemma \ref{lem-C1C2G1G2} we immediately obtain 
\begin{corollary}
Either $G$ is of product type, or $m, m'\in\{1,2,3,4,6,8\}$. 
\end{corollary}

\begin{lem}\label{G splits}
If both exact sequences \eqref{exact-sequence-restriction-to-K3} do not split, then $G$ is of product type.
\end{lem}


\begin{proof}
By symmetry, it is enough to prove the result for the first exact sequence in \eqref{exact-sequence-restriction-to-K3}. We have $G=G_{p_1}\times\dots\times G_{p_k}$, and $G_{p_i}$ is an abelian group which fits in a short exact sequence of the form
\[
	0\to C_{p_i}\to G_{p_i}\to H_{p_i}\to0,
\]
where $C_{p_i}$ and $H_{p_i}$ are $p_i$-Sylow subgroups of $C$ and $H$, respectively. 
We will proceed as follows.
\begin{enumerate}
	\item We fix an isomorphism class for $G$, among the groups (1)--(6) in Theorem \ref{thm-maximal-k3-groups}.
	\item For each $m\in\{1,2,3,4,6,8\}$, we use Theorem \ref{littelwood} to compute all possible classes of $p$-Sylow subgroups $G_p$ of $G$.
	\item We deduce the possible isomorphism classes for $G$.
\end{enumerate}
Applying systematically this method, we obtain the following. We have $C=\mathbb{Z}/m$. Assume that $G$ is not isomorphic to $\Z/m\times H$ and that $\mathfrak r(G)>3$ (recall that if $\mathfrak r(G)\leq 3$ then $G$ is of product type). Then for $m\in\{1,2,3,4,6,8\}$ we have the following possibilities.
\begin{itemize}
        \item $H=(\Z/4)^3$. Then we have the following possibilities
        \begin{enumerate}
        \item 
        $m=4$,  $G=\Z/8\times(\Z/4)^2\times\Z/2$, \item 
        $m=8$,  $G=\Z/16\times(\Z/4)^2\times\Z/2$. 
        \end{enumerate}
        
        \item $H=\Z/8\times\Z/4\times\Z/2$. Then we have the following
        \begin{enumerate}
        \item 
        $m=4$, 
        $G=\Z/16\times\Z/4\times(\Z/2)^2$, or $G=(\Z/8)^2\times(\Z/2)^2$,
        \item 
        $m=8$, $G=\Z/32\times\Z/4\times(\Z/2)^2$, or $G=\Z/16\times\Z/8\times(\Z/2)^2$, or $G=\Z/16\times(\Z/4)^2\times\Z/2$. 
        \end{enumerate}
        
        \item $H=\Z/4\times(\Z/2)^3$. Then
        \begin{enumerate}
        \item 
        $m=2$, $G=\Z/8\times(\Z/2)^3$, or $G=(\Z/4)^2\times(\Z/2)^2$,
        \item 
        $m=4$, $G=\Z/16\times(\Z/2)^3$, or $G=\Z/8\times(\Z/2)^4$, or $G=\Z/8\times\Z/4\times(\Z/2)^2$,
        \item 
        $m=6$, $G=\Z/24\times(\Z/2)^3$,
        \item 
        $m=8$, $G=\Z/32\times(\Z/2)^3$, or $G=\Z/16\times\Z/4\times(\Z/2)^2$, or $G=\Z/16\times(\Z/2)^4$. 
        \end{enumerate}
        
        \item $H=\Z/2\times(\Z/3)^3$. In this case we always have $\mathfrak r(G)\le3$.
        \item  $H=(\Z/3)^2\times(\Z/2)^3$. We obtain that $\mathfrak r(G)\le3$, or that $G=\Z/m\times H$.
        \item $H=(\Z/2)^5$.  $G=\Z/2^i\times(\Z/2)^4$ with $2\leq i\leq 4$, or $G=\Z/12\times(\Z/2)^4$.
    \end{itemize}
Among all these possibilities, we see that if $G$ is not isomorphic to $\Z/m\times H$, then $G$ is of product type.
\end{proof}

\begin{thm}
\label{thm-h0-geq-2}
Assume that $X$ is a $G\mathbb{Q}$-Fano threefold with $h^0(-K_X)\geq 2$ where $G$ is a finite abelian group. 
If $G$ is of K3 type and not of product type then $G$ is isomorphic to one of the following groups: 
\begin{enumerate}
\item 
$(\mathbb{Z}/4)^4$,
\item 
$(\mathbb{Z}/8)^2\times\mathbb{Z}/4\times\mathbb{Z}/2$,
\item
$(\mathbb{Z}/6)^2\times(\mathbb{Z}/3)^2$,
\item
$(\mathbb{Z}/6)^3\times \mathbb{Z}/2$.
\end{enumerate} 
\end{thm}
\begin{proof}
By Lemma \ref{G splits}, we may assume that $G= C\times H$, where $C=\Z/m$ and $H$ is one of the groups (1)--(6) as in Theorem \ref{thm-maximal-k3-groups}. 
Moreover, by symmetry we have that $G$ is isomorphic to $C'\times H'$ where $C'=\Z/m'$, and by Lemma \ref{lem-C1C2G1G2} we get $C\subset H'$ and $C'\subset H$. We will proceed in the following way.
\begin{enumerate}
    \item For a given group $H$ from the groups (1)--(6) in Theorem \ref{thm-maximal-k3-groups}, we deduce the possibilities for $C'$.
    \item For a given pair $(H,C')$, we find all the possibilities for $m$.
\end{enumerate}
We obtain the following possible configurations.
\begin{itemize}
    \item $H=(\Z/4)^3$.
    \begin{longtable}{|c|c|c|c|}\hline
        $C'$&$\tg$&$\Z/2$&$\Z/4$\\ \hline
        $m$&1&1,2&1,2,4\\ \hline
    \end{longtable}
    \noindent The group $G$ is isomorphic to a subgroup of $(\Z/4)^4$.
    \item $H=\Z/8\times\Z/4\times\Z/2$.
    \begin{longtable}{|c|c|c|c|c|}\hline
        $C'$&$\tg$&$\Z/2$&$\Z/4$&$\Z/8$\\ \hline
        $m$&1&1,2&1,2,4&1,2,4,8\\ \hline
    \end{longtable}
    \noindent The group $G$ is isomorphic to a subgroup of $(\Z/8)^2\times\Z/4\times\Z/2$.
    \item $H=\Z/4\times(\Z/2)^3$.
    \begin{longtable}{|c|c|c|c|}\hline
        $C'$&$\tg$&$\Z/2$&$\Z/4$\\ \hline
        $m$&1&1,2&1,2,4\\ \hline
    \end{longtable}
    \noindent The group $G$ is isomorphic to a subgroup of $(\Z/4)^2\times(\Z/2)^3$, which is of product type.
    \item $H=\Z/6\times(\Z/3)^2$.
    \begin{longtable}{|c|c|c|c|c|}\hline
        $C'$&$\tg$&$\Z/2$&$\Z/3$&$\Z/6$\\ \hline
        $m$&1&1,2&1,3&1,2,3,6\\ \hline
    \end{longtable}
    \noindent The group $G$ is isomorphic to a subgroup of $(\Z/6)^2\times(\Z/3)^3$.
    \item $H=(\Z/6)^2\times\Z/2$.
    \begin{longtable}{|c|c|c|c|c|}\hline
        $C'$&$\tg$&$\Z/2$&$\Z/3$&$\Z/6$\\ \hline
        $m$&1&1,2&1,3&1,2,3,6\\ \hline
    \end{longtable}
    \noindent The group $G$ is isomorphic to a subgroup of $(\Z/6)^3\times\Z/2$.
    \item $H=(\Z/2)^5$.
    \begin{longtable}{|c|c|c|}\hline
        $C'$&$\tg$&$\Z/2$\\ \hline
        $m$&1&1,2\\ \hline
    \end{longtable}
    \noindent The group $G$ is isomorphic to a subgroup of $(\Z/2)^6$, which is of product type.
\end{itemize}
\end{proof}

\section{Orbits of non-Gorenstein points}
\label{sec-orbits}
In this section, we study the orbits of non-Gorenstein points under the action of a finite abelian group. 
Let $G$ be a finite abelian group, and let $X$ be a $3$-dimensional $G\mathbb{Q}$-Fano variety. 
Assume that the set of non-Gorenstein singularities of $X$ is non-empty. 
Let us denote the set of all non-Gorenstein points of $X$ as follows: 
\begin{equation}
\label{eq-baskets}
k_1 \times P_1,\quad \quad k_2\times P_2,\quad \quad \dots,\quad \quad k_l\times P_l,\quad \quad \quad k_i\geq 1
\end{equation}
for $l\geq 1$ where each $P_i\in X$ is a germ of a terminal non-Gorenstein singularity of index $r_i>1$, and $k_i\times P_i$ means that we have exactly $k_i$ singular points of $X$  locally analytically isomorphic to $P_i\in X$.  
In particular, $P_i\in X$ and $P_j\in X$ are not locally analytically isomorphic for $i\neq j$. Hence, each set $\{k_i\times P_i\}$ for $1\leq i\leq l$ splits into $G$-orbits.

Assume that the pair $(X, S)$ is plt where $S\in |-K_X|$ is a $G$-invariant element. Then $S$ is a K3 surface with at worst canonical singularities. We have the following exact sequence:
\begin{equation}
\label{eq-exact-seq-k3-type}
0\to \mathbb{Z}/m\to G \to H \to 0
\end{equation}
where $H$ faithfully acts on $S$, and $m\geq 1$. 
By Lemma \ref{lem-singular-du-val}, we see that $S$ has at least $k_1+\ldots+k_l$ du Val singularities that correspond to non-Gorenstein singularities on $X$. 
We denote by $f\colon S'\to S$ the $H$-equivariant minimal resolution of $S$, so that $S'$ is a smooth K3 surface with a faithful action of~$H$. According to Remark \ref{rem-k3-or-product}, either $G$ is of product type, or $H$ is one of the groups (1)--(6) as in Proposition \ref{prop: determinants}. We examine them case by case.  

\begin{lem}
\label{lem-exclusion-of-two-groups}
The group $H$ in \eqref{eq-exact-seq-k3-type} is not isomorphic either to $(\mathbb{Z}/6)^2\times\mathbb{Z}/2$ or to $(\mathbb{Z}/4)^3$.
\end{lem}
\begin{proof}
By Corollary \ref{cor-rank-1-is-smooth}, in this case $S$ is a smooth K3 surface. However, this contradicts the assumption $l\geq 1$ and Lemma \ref{lem-singular-du-val}. 
\end{proof}

\begin{lem}\label{Z842 orbits}
	If $H=\dz/8\times\dz/4\times\dz/2$ and $G$ is not of product type, then $l=1$ and   
	$k_1$ is even. Morever, the action of $H$ (as well as $G$) on the set $\{k_i\times P_1\}$ is transitive. In particular, $k_1$ divides~$64$. 
\end{lem}
\begin{proof}
By Corollary \ref{cor-rank-2-is-transitive}, there exists at most one $H$-orbit of singular points of $S$. 
By Lemma \ref{lem-singular-du-val} it follows that there exists at most one $G$-orbit of non-Gorenstein singular points on $X$. 
Thus, we have $l=1$, and the action of $G$ on the set $\{k_i\times P_1\}$ is transitive. 
If $k_1$ is odd, then $G$ has a fixed point, so by Theorem \ref{thm-action-on-terminal-point} we have that $G$ is of product type. The last claim follows from the orbit-stabilizer theorem.
\end{proof}

\begin{lem}\label{Z633 orbits}
If $H=\dz/6\times(\dz/3)^2$ and $G$ is not of product type, then $l=1$, and %
$k_1$ is divisible by~$3$. 
Moreover, the action of $H$ (as well as $G$) on the set $\{k_i\times P_1\}$ is transitive. In particular,~$k_1$ divides $54$. 
\end{lem}
\begin{proof}
We denote by $G_3$ the $3$-Sylow subgroup of $G$. 
As in the proof of Lemma \ref{Z842 orbits}, by Lemma \ref{lem-singular-du-val} and Corollary \ref{cor-rank-2-is-transitive} we have $l=1$, and the action of $G$ on the set $\{k_i\times P_1\}$ is transitive. 
If $k_1$ is not divisible by $3$, then $G_3$ has a fixed point. We deduce from Theorem \ref{thm-action-on-terminal-point} that the $3$-Sylow subgroup $G_3$ is of rank $3$. For $p\neq 3$, since $\mathfrak{r}(H_p)\leq 1$, we have $\mathfrak{r}(G_p)\leq 2$. We deduce that $\mathfrak{r}(G)\leq 3$. Thus $G$ is of product type. The last claim follows from the orbit-stabilizer theorem.
\end{proof}

\begin{lem}\label{Z22222 orbits}
	If $H=(\dz/2)^5$ and $G$ is not of product type, then 
	the non-Gorenstein locus of $X$ is composed of orbits whose lengths are multiples of $8$. In particular, we have that 
	$k_i$ is divisible by $8$ for all $i\in\{1,\dots,l\}$.
\end{lem}
\begin{proof}
We denote by $G_2$ the $2$-Sylow subgroup of $G$. 
	If there exists $i\in\{1,\dots,l\}$ such that $k_i$ is not divisible by $8$, then $G_2$ has an orbit of length $1,2,$ or $4$. But since $H=(\dz/2)^5$, the group $G_2$ is of the form $\dz/2^k\times(\dz/2)^5$, for some $k\geq1$, or $\dz/2^k\times(\dz/2)^4$, for some $k\geq1$. Hence, all subgroups of $G_2$ of index $1$ or $2$ have rank at least $4$, and Theorem \ref{thm-action-on-terminal-point} implies that $G_2$ cannot have an orbit of length $1$ or $2$. Assume there exists an orbit of length $4$. Then $G_2$ has a subgroup of rank $3$ or $4$ fixing a point, and the latter is excluded by the same result. So $G_2$ has a subgroup of index $4$ and rank $3$, which implies that $G_2$ is of the form $\dz/2^k\times(\dz/2)^4$, for some $k\geq1$. But then $G$ is of product type.
\end{proof}

\begin{lem}\label{Z4222 orbits}
	If $H=\dz/4\times(\dz/2)^3$ and $G$ is not of product type, then the non-Gorenstein locus of $X$ is composed of orbits whose lengths are multiples of $4$. In particular, we have that 
	$k_i$ is divisible by $4$ for all $i\in\{1,\dots,l\}$.
\end{lem}
\begin{proof}
We denote by $G_2$ the $2$-Sylow subgroup of $G$. 
	If there exists $i\in\{1,\dots,l\}$ such that $k_i$ is not divisible by $4$, then $G_2$ has an orbit of length $1$ or $2$. The first case is excluded by Theorem \ref{thm-action-on-terminal-point}, since $r(G_2)\geq r(H)=4$. If $G_2$ has a subgroup $G_2'$ of index 2 fixing a point, then $G_2'$ is of rank $3$, and we deduce that either $r(G_2)\leq 3$, and hence $G$ is of product type, or $G_2=\dz/2^k\times\dz/4\times(\dz/2)^2$ for $k\geq1$, or $G_2=\dz/2^k\times(\dz/2)^3$, with $k\geq2$. In all cases, the group $G$ is of product type.
\end{proof}

We obtain the following.

\begin{corollary}\label{coro:coprime singularities}

Either $G$ is of product type, or $\mathrm{gcd}(k_1,\dots,k_l)$ is divisible either by $2$ or by $3$, where $k_i$ are as in \eqref{eq-baskets}. In particular, if $\mathrm{gcd}(k_1,\dots,k_l)=1$ then $G$ is of product type.
\end{corollary} 

Now let us denote all the cyclic quotient singularities in the basket of $X$ (see Section \ref{sec-about-baskets}) as follows: 
\begin{equation}
t_1 \times Q_1,\quad \quad t_2\times Q_2,\quad \quad \dots,\quad \quad t_s\times Q_s,\quad \quad \quad t_i\geq 1, 
\end{equation}
for $s\geq 1$ where each $Q_i\in X$ is a germ of a cyclic quotient singularity of index $r'_i>1$, and $t_i\times Q_i$ means that we have exactly $t_i$ singular points of $X$ locally analytically isomorphic to $Q_i\in X$. 
We assume that $Q_i$ and $Q_j\in X$ are not locally analytically isomorphic for $i\neq j$. 
From Corollary \ref{coro:coprime singularities} we obtain the following.
\begin{corollary}\label{coro:coprime singularities2}
Either $G$ is of product type, or $\mathrm{gcd}(t_1,\dots,t_l)$ is divisible either by $2$ or by~$3$. In particular, if $\mathrm{gcd}(t_1,\dots,t_l)=1$ then $G$ is of product type.
\end{corollary} 

After applying 
Corollary \ref{coro:coprime singularities2} as well as Lemmas \ref{lem-exclusion-of-two-groups}--\ref{Z4222 orbits} to all the Fano threefolds of Fano index~$1$ and Fano genus $-1$, which we went through using the Graded Ring Database \cite{GRDB}, and taking off the baskets consisting only of singularities of type $\frac{1}{2}(1,1,1)$ which will be treated in Section \ref{sec-Case $h^0(-K_X)=1$}, we end up with the following result.
\begin{proposition}\label{prop: 91 poss}
Let $X$ be a $G\mathbb{Q}$-Fano threefold where $G$ is a finite abelian group. 
Assume that $h^0(-K_X)=1$. 
If $G$ is not of product type, then the Fano index of $X$ is at most $4$. Furthermore,  
either the basket of $X$ only consists of points of type $\frac{1}{2}(1,1,1)$, or its basket is among the following possibilities.
\end{proposition}

\begingroup
\renewcommand{\arraystretch}{1.2}
\begin{longtable}{|p{0.45\textwidth}|p{0.45\textwidth}|}
\hline
\textbf{Basket of singularities of $X$} & \textbf{Possibilities for $H$} \\ \hline
\endhead
$2 \times \frac{1}{10}(3,7,1)$ &  $\dz/8\times\dz/4\times\dz/2$ \\ \hline
$2 \times \frac{1}{11}(4,7,1)$ &  $\dz/8\times\dz/4\times\dz/2$ \\ \hline
$6 \times \frac{1}{4}(1,3,1)$ &  $\dz/8\times\dz/4\times\dz/2$, \newline $\dz/6\times(\dz/3)^2$ \\ \hline
$2 \times \frac{1}{9}(2,7,1)$ &  $\dz/8\times\dz/4\times\dz/2$ \\ \hline
$6 \times \frac{1}{2}(1,1,1), 2 \times \frac{1}{4}(1,3,1)$ &  $\dz/8\times\dz/4\times\dz/2$  \\ \hline
$4 \times \frac{1}{2}(1,1,1), 4 \times \frac{1}{3}(1,2,1)$ & $\dz/4\times(\dz/2)^3$ \\ \hline
$4 \times \frac{1}{5}(2,3,1)$ &  $\dz/8\times\dz/4\times\dz/2$, \newline $\dz/4\times(\dz/2)^3$ \\ \hline
$4 \times \frac{1}{2}(1,1,1), 4 \times \frac{1}{4}(1,3,1)$ &  $\dz/8\times\dz/4\times\dz/2$, \newline $\dz/4\times(\dz/2)^3$ \\ \hline
$2 \times \frac{1}{11}(3,8,1)$ &  $\dz/8\times\dz/4\times\dz/2$ \\ \hline
$8 \times \frac{1}{3}(1,2,1)$ & $(\dz/2)^5$, \newline $\dz/4\times(\dz/2)^3$, \newline $\dz/8\times\dz/4\times\dz/2$  \\ \hline
$3 \times \frac{1}{7}(2,5,1)$ & $\dz/6\times(\dz/3)^2$ \\ \hline
$3 \times \frac{1}{7}(3,4,1)$ & $\dz/6\times(\dz/3)^2$ \\ \hline
$6 \times \frac{1}{2}(1,1,1), 3 \times \frac{1}{4}(1,3,1)$ & $\dz/6\times(\dz/3)^2$  \\ \hline
$2 \times \frac{1}{11}(2,9,1)$ &  $\dz/8\times\dz/4\times\dz/2$ \\ \hline
$8 \times \frac{1}{2}(1,1,1), 2 \times \frac{1}{4}(1,3,1)$ &  $\dz/8\times\dz/4\times\dz/2$  \\ \hline
$10 \times \frac{1}{2}(1,1,1), 2 \times \frac{1}{4}(1,3,1)$ &  $\dz/8\times\dz/4\times\dz/2$  \\ \hline
$8 \times \frac{1}{2}(1,1,1), 4 \times \frac{1}{3}(1,2,1)$ & $\dz/4\times(\dz/2)^3$  \\ \hline
\end{longtable}
\begin{center}
\emph{Table 6. Possible baskets of singularities on $X$ of index $1$ and corresponding groups $H$}
\end{center}
\endgroup
\begingroup
\renewcommand{\arraystretch}{1.2}
\begin{longtable}{|p{0.45\textwidth}|p{0.45\textwidth}|}
\hline
\textbf{Basket of singularities of $X$} & \textbf{Possibilities for $H$} \\ \hline
\endhead
$2 \times \frac{1}{3}(1,2,2), 2 \times  \frac{1}{7}(3,4,2)$ & $\dz/8\times\dz/4\times\dz/2$\\ \hline
$4 \times \frac{1}{3}(1,2,2), 2 \times \frac{1}{5}(1,4,2)$ & $\dz/8\times\dz/4\times\dz/2$ \\ \hline
$2 \times \frac{1}{5}(2,3,2), 2 \times \frac{1}{7}(1,6,2)$ & $\dz/8\times\dz/4\times\dz/2$ \\ \hline
$2 \times \frac{1}{11}(4,7,2)$ & $\dz/8\times\dz/4\times\dz/2$ \\ \hline
$2 \times \frac{1}{5}(1,4,2), 2 \times \frac{1}{7}(3,4,2)$ & $\dz/8\times\dz/4\times\dz/2$ \\ \hline
$3 \times \frac{1}{3}(1,2,2), 3 \times \frac{1}{5}(1,4,2)$ & $\dz/6\times(\dz/3)^2$ \\ \hline
$3 \times \frac{1}{7}(3,4,2)$ & $\dz/6\times(\dz/3)^2$ \\ \hline
$2 \times \frac{1}{3}(1,2,2), 2 \times \frac{1}{9}(4,5,2)$ & $\dz/8\times\dz/4\times\dz/2$ \\ \hline
\end{longtable}
\begin{center}
\emph{Table 7. Possible baskets of singularities on $X$ of index $2$ and corresponding groups $H$}
\end{center}
\begingroup
\renewcommand{\arraystretch}{1.2}
\begin{longtable}{|p{0.45\textwidth}|p{0.45\textwidth}|}
\hline
\textbf{Basket of singularities of $X$} & \textbf{Possibilities for $H$} \\ \hline
\endhead
$4 \times \frac{1}{5}(1,4,3)$ & $\dz/8\times\dz/4\times\dz/2$, \newline $\dz/4\times(\dz/2)^3$\\ \hline
$2 \times \frac{1}{2}(1,1,1), 2 \times \frac{1}{8}(1,7,3)$ & $\dz/8\times\dz/4\times\dz/2$ \\ \hline
$4 \times \frac{1}{2}(1,1,1), 2 \times \frac{1}{7}(1,6,3)$ & $\dz/8\times\dz/4\times\dz/2$ \\ \hline
\end{longtable}
\begin{center}
\emph{Table 8. Possible baskets of singularities on $X$ of index $3$ and corresponding groups $H$}
\end{center}
\begingroup
\renewcommand{\arraystretch}{1.2}
\begin{longtable}{|p{0.45\textwidth}|p{0.45\textwidth}|}
\hline
\textbf{Basket of singularities of $X$} & \textbf{Possibilities for $H$} \\ \hline
\endhead
$2 \times \frac{1}{9}(2,7,4)$ & $\dz/8\times\dz/4\times\dz/2$\\ \hline
$3 \times \frac{1}{7}(1,6,4)$ & $\dz/6\times(\dz/3)^2$ \\ \hline
$4 \times \frac{1}{3}(1,2,1), 2 \times \frac{1}{5}(1,4,4)$ & $\dz/8\times\dz/4\times\dz/2$ \\ \hline
$2 \times \frac{1}{5}(2,3,4), 2 \times \frac{1}{7}(1,6,4)$ & $\dz/8\times\dz/4\times\dz/2$ \\ \hline
\end{longtable}
\begin{center}
\emph{Table 9. Possible baskets of singularities on $X$ of index $4$ and corresponding groups $H$}
\end{center}
\begin{proof}
When the Fano index is 1 or 2, all possibilities of the baskets of singularities of a genus -1 Fano threefold are classified in the Graded Ring Database \cite{GRDB}, using the method described in \cite{BK22}. Going through the possibilities of baskets and apply Lemmas~\ref{Z842 orbits} to~\ref{Z4222 orbits}, we obtain Tables 6 and 7. 

When the Fano index is greater than 2, we apply numerical conditions as in \cite[Section 2]{BK22} to find all possible baskets of singularities, use the lower bound for the genus as in \cite[Corollary 2.7]{BK22} to find possibilities for those which can be realized on genus -1 Fano threefolds, and Lemmas~\ref{Z842 orbits} to~\ref{Z4222 orbits} to obtain Tables 8 and 9. In particular, we find that the Fano index of $X$ is possibly at most 4. Note that this does not guarantee the existence of such Fano threefolds with singularities listed in Tables 8 or 9.  

We include the {\tt magma} code to produce Tables 6,7,8, and 9 in \cite{Z-web}.
\end{proof}

\begin{proposition}
\label{prop-index2-exclusion}
If $X$ has index $2$, then $G$ is of product type.
\end{proposition}
\begin{proof}
Since the number of points of the same type in the basket of $X$ in each case of Table $7$ is at most $4$, we get that either the non-Gorenstein points on $X$ are cyclic quotient singularities, or $G$ has a fixed point on $X$. In the latter case, $G$ is of product type, hence we may assume that the former case is realized. 

Using Corollary \ref{cor-s-s-tilde}, we may assume that  a cyclic quotient singularity $x\in X$ of index $r$ corresponds to a singular point $x\in S$ of type $A_{r-1}$ on $S$. 
Note that in each case of Table~$7$ we have $r>2$. 
On the other hand, the group $H$ is isomorphic either to $\mathbb{Z}/8\times \mathbb{Z}/4\times \mathbb{Z}/2$, or to $\dz/6\times(\dz/3)^2$.
But this is a contradiction with Corollary \ref{cor-rank-2-is-transitive}.
\end{proof}

\begin{proposition}
\label{prop-index34-exclusion}
If $X$ has index $3$ or $4$, then $G$ is of product type.
\end{proposition}
\begin{proof}
Note that the all the baskets from Table 9 also can be found in Tables 6 and 7 (possibly, with different weights of the cyclic group action). Since our arguments in fact do not use the  values of the weights, we see that in the index $4$ case $G$ is of product type. Similarly, the first case in Table 8 can be found in Table 6 (again, up to changing the weights). Two remaining cases in Table 8 are dealt with similarly to the proof of Proposition \ref{prop-index2-exclusion}.
\end{proof}

\section{Case $h^0(-K_X)=1$ with half-points}
\label{sec-Case $h^0(-K_X)=1$}
To illustrate our approach, we treat the case when the non-Gorenstein points of $X$ have type $\frac{1}{2}(1,1,1)$. Informally, we call them \emph{half-points}. 
The main goal of this section is to prove the following
\begin{thm}
\label{thm-half-points}
Let $X$ be a $G\mathbb{Q}$-Fano threefold where $G$ is a finite abelian group. Assume that $h^0(-K_X)=1$. Also, assume that the non-Gorenstein locus of $X$ consists only of points of type $\frac{1}{2}(1,1,1)$. Then $G$ is of product type.
\end{thm}

\begin{proposition}\label{half}
Assume that all the terminal non-Gorenstein points of $X$ have type $\frac{1}{2}(1,1,1)$. Then $9 \leq N\leq 15$. 
\end{proposition}
\begin{proof}
Since $h^0(X, -K_X)=1$, we see that $X$ is non-Gorenstein. For $D = -K_X$ and singular points of type $\frac{1}{2}(1,1,1)$ we have $c_Q = -1/8$, so if there are $N$ such points, in the formula \eqref{eq-ORR} they give a contribution $-N/8$. From \eqref{eq-Miyaoka} it follows that 
\begin{equation}
(-K_X)\cdot c_2(X) = 24 - \frac{3N}{2}.
\end{equation}
We have
\begin{equation}
\label{eq-rr-half-points}
1=h^0 ( \OOO_X (-K_X) ) = 3 + \frac{1}{2}\left(-K_X\right)^3 - \frac{N}{4} 
\end{equation}
Using $(-K_X)^3\geq \frac{1}{2}$ we obtain 
$N \geq 9$. 
Since by \eqref{eq-Miyaoka} we have $N\leq 15$, the claim follows. 
\end{proof}


\begin{proposition}
\label{prop-N=9}
    If $N=9$ then $G$ is of product type.
\end{proposition}
\begin{proof} 
By Lemmas \ref{lem-exclusion-of-two-groups}--\ref{Z4222 orbits} we see that $H=\dz/6\times(\dz/3)^2$. 
Moreover, the $9$ singular points of $S$ that correspond to the $9$ singular points on $X$ of type $\frac{1}{2}(1,1,1)$ form one $G$-orbit of length $9$. Let $f\colon S'\to S$ be a $H$-equivariant minimal resolution of $S$. By Corollary \ref{cor-rank-2-is-transitive}, the surface $S$ has singularities of type $A_1$ or $A_2$. 
In both cases there is an $f$-exceptional $G$-invariant curve of self-intersection $-18$ on $ S'$.
The index $\mu$ of the class group $\mathrm{Cl}(S)$ in $\mathrm{NS}(S)$ is equals either $2$ or $3$.  By Proposition~\ref{prop-inclusion-of-lattices}, we have
\begin{equation}
\mu\mathrm{NS}({S'})^H\subset f^*(\mathrm{NS}(S)^H)\oplus V\subset \mathrm{NS}({S'})^H.
\end{equation}

Assume that $\mu=2$. 
By Proposition \ref{prop-ranks-H0} we have
\[
\left(2\mathrm{NS}({S'})^H, \begin{pmatrix}0&12\\12&0\end{pmatrix} \right)\subset 
\left(f^*(\mathrm{NS}(S)^H)\oplus V, \begin{pmatrix}a&0\\0&-18\end{pmatrix} \right)
\subset
\left(\mathrm{NS}({S'})^H, \begin{pmatrix}0&3\\3&0\end{pmatrix}\right).
\]
Note that $a=6a'$, since the values of the intersection form on $\mathrm{NS}({S'})^H$ is divisible by $6$. By \eqref{eq-determinant-of-lattices} we have that $-108a'$ divides $-144$ for $a'\geq 1$ which is a contradiction. 

Assume that $\mu=3$. 
Analogously to the previous case, we have
\[
\left(2\mathrm{NS}({S'})^H, \begin{pmatrix}0&27\\27&0\end{pmatrix} \right)\subset 
\left(f^*(\mathrm{NS}(S)^H)\oplus V, \begin{pmatrix}a&0\\0&-18\end{pmatrix} \right)
\subset
\left(\mathrm{NS}({S'})^H, \begin{pmatrix}0&3\\3&0\end{pmatrix}\right).
\]
Then $-18a$ divides $-27^2$, which is a contradiction. This shows that $G$ is of product type.
\end{proof}

\begin{proposition}
\label{prop-k3-with-orbit-4}
Let $X$ be a $G\mathbb{Q}$-Fano threefold such that all non-Gorenstein singular points of $X$ have type $\frac{1}{2}(1,1,1)$.
Assume that $S\in |-K_X|$ is a $G$-invariant K3 surface with at worst du Val singularities. Assume that 
\begin{enumerate}
\item
$H=\mathbb{Z}/4\times (\mathbb{Z}/2)^3$ where $H$ is as in \eqref{eq-exact-seq-k3-type},
\item
there exists a $G$-orbit of points of type $\frac{1}{2}(1,1,1)$ of cardinality $4$. 
\end{enumerate}
Then $G$ is of product type. 
\end{proposition}
\begin{proof}

By Corollary \ref{cor-s-s-tilde} and using its notation, we may assume that $\widetilde{x}\in \widetilde{S}$ is smooth. 
Let $\Sigma = \{x=x_1,x_2,x_3,x_4\}$ be the $G$-orbit of $x$ (or, equivalently, its $H$-orbit). Consider the following exact sequences: 
\begin{equation}
\label{seq-x-sigma2}
0 \to G_x \to G\to G_\Sigma \to 0,
\end{equation}
\begin{equation}
\label{seq-x-sigma}
0 \to H_x \to H \to H_\Sigma \to 0.
\end{equation}
where $G_\Sigma$ (respectively, $H_\Sigma$) is the image of $G$ (respectively, of $H$) in the group of permutations of $\Sigma$, and $G_x$ (respectively, $H_x$) are stabilizers of $x$ in $G$ (respectively, $H$). Note that $G_\Sigma=H_\Sigma$. 

We will show that in the assumption of the theorem and assuming that $\widetilde{x}\in\widetilde{S}$ is smooth, the $2$-Sylow subgroup $G_2$ cannot faithfully act on $X$. So replacing $G$ by $G_2$ we will assume that $G$ is a $2$-group.  

By Proposition
\ref{prop-lift-of-2-group-is-abelian} we see that $\widetilde{G}_x$ is abelian, and hence $\widetilde{H}_x$ is abelian. 
Since $\widetilde{x}\in \widetilde{S}$ is smooth, in the notation of diagram \eqref{diagram-two-levels-of-groups}, we see that $\mathfrak{r}(\widetilde{H}_x)\leq 2$ and so $\mathfrak{r}(H_x)\leq 2$. 
Since the action of $H$ on $\Sigma$ is transitive, we have that either $H_\Sigma=\mathbb{Z}/4$, or $H_\Sigma=(\mathbb{Z}/2)^2$. 
The first possibility $H_\Sigma=\mathbb{Z}/4$, $H_x = (\mathbb{Z}/2)^3$ is not realized as $\mathfrak{r}(H_x)\leq 2$. Hence we have  $H_\Sigma=(\mathbb{Z}/2)^2$, and either $H_x = \mathbb{Z}/4\times \mathbb{Z}/2$, or $H_x=(\mathbb{Z}/2)^3$. Similarly, the second subcase is not realized since $\mathfrak{r}(H_x)\leq 2$. We obtain $H_\Sigma=(\mathbb{Z}/2)^2$, $H_x = \mathbb{Z}/4\times \mathbb{Z}/2$. In particular, the exact sequence \eqref{seq-x-sigma} splits: $H=H_x\times H_\Sigma$. 

There exists the following diagram:
\begin{equation}
\begin{tikzcd}
\widetilde{E}\subset \mathrm{Bl}_0\mathbb{C}^2 \ar[d] \ar[r] & \mathrm{Bl}_x{S} \supset E \ar[d]  \\
0\in \mathbb{C}^2 \ar[r] & S \ni x
\end{tikzcd}
\end{equation}
where $\mathrm{Bl}_0\mathbb{C}^2$ is the blow up of $\mathbb{C}^2$ at the closed point $0$,  $\mathrm{Bl}_x{S}$ is the blow up of $S$ at the closed point $x$, and horizontal arrows are quotient maps by the action of $\mathbb{Z}/2$. 
Hence, $\widetilde{H}_x$ faithfully acts on $\mathrm{Bl}_0 \mathbb{C}^2$, and $H_x$ faithfully acts on $\mathrm{Bl}_x S$. There are exact sequences:
\begin{equation}
\label{exact-sequence-E}
0 \to H_N \to H_x \to H_{E} \to 0,
\end{equation}
\begin{equation}
\label{exact-sequence-E-tilde}
0 \to \widetilde{H}_N \to \widetilde{H}_x \to \widetilde{H}_{\widetilde{E}} \to 0,
\end{equation}
where 
\begin{enumerate}
\item
$H_{E}$ faithfully acts on the $(-2)$-curve $E$ on $\mathrm{Bl}_x{S}$, 
\item
$\widetilde{H}_{\widetilde{E}}$ faithfully acts on the $(-1)$-curve $\widetilde{E}$ on $\mathrm{Bl}_0\mathbb{C}^2$,
\item
$H_N$ faithfully acts in the normal bundle to $E$ on $\mathrm{Bl}_x{S}$,
\item
$\widetilde{H}_N$ faithfully acts in the normal bundle to $\widetilde{E}$ on $\mathrm{Bl}_0\mathbb{C}^2$.
\end{enumerate}
Since $\widetilde{E}$ is the ramification curve of the quotient map $\mathrm{Bl}_0\mathbb{C}^2\to\mathrm{Bl}_x S$, we have $H_{E}=\widetilde{H}_{\widetilde{E}}$, and $\widetilde{H}_N/(\mathbb{Z}/2)=H_N$. 
From diagram \eqref{diagram-two-levels-of-groups} we obtain an  exact sequence:
\begin{equation}
0 \to \mathbb{Z}/2 \to \widetilde{H}_x \to H_{x} \to 0.
\end{equation}

Since $\mathfrak{r}(\widetilde{H}_x)=2$, $\widetilde{H}_x$ is abelian and $H_x=\mathbb{Z}/4\times\mathbb{Z}/2$, we have that either $\widetilde{H}_x = \mathbb{Z}/8\times \mathbb{Z}/2$, or $\widetilde{H}_x = (\mathbb{Z}/4)^2$.


Assume that $\widetilde{H}_x = (\mathbb{Z}/4)^2$. Then $\widetilde{H}_{\widetilde{E}}=H_E=\mathbb{Z}/4$ whose generator acts on $\mathbb{C}^2$ via the matrix $\begin{pmatrix}\sqrt{-1}&0\\0&-\sqrt{-1}\end{pmatrix}$, $\widetilde{H}_N=\mathbb{Z}/4$, whose generator acts on $\mathbb{C}^2$ via the matrix $\begin{pmatrix}\sqrt{-1}&0\\0&\sqrt{-1}\end{pmatrix}$. We have $H_N=\mathbb{Z}/2$. 
Hence, the exact sequence \eqref{exact-sequence-E} splits. 
From the local description it follows that $\widetilde{H}_E=H_E$ preserves the form $dx\wedge dy$ on $\mathbb{C}^2$ which descends to $x\in S$. 
It follows that $H_E$  
acts symplectically on $S$. 
Hence, there exists a symplectic element of order $4$ in $H=\mathbb{Z}/4\times(\mathbb{Z}/2)^3$. However, this contradicts Corollary~\ref{cor-max-group-direct-product}.

Thus we have $\widetilde{H}_x = \mathbb{Z}/8\times \mathbb{Z}/2$. Let $\alpha$ and $\beta$ be the generators of the first and the second factors of $\widetilde{H}_x = \mathbb{Z}/8\times \mathbb{Z}/2$. 
Since $\widetilde{H}_x/(\mathbb{Z}/2)=H_x=\mathbb{Z}/4\times\mathbb{Z}/2$, and we quotient by the element $\sigma$ which acts on $\mathbb{C}^2$ via multiplication by $-1$, it follows that $4\alpha$ is equal to $\sigma$. Hence we may assume that $\beta$ is given by the matrix $\begin{pmatrix}{1}&0\\0&{-1}\end{pmatrix}$. However, this is a contradiction, because in this case $\beta$ descends to an element which acts non-symplectically on $S$. Indeed, by Corollary \ref{cor-max-group-direct-product} the element $\alpha$ descends to a non-symplectic automorphism of order $4$ on $S$, and hence the image of $\alpha$ should act on $S$ symplectically. This completes the proof.
\end{proof}

\begin{proof}[Proof of Theorem \ref{thm-half-points}]   We work in the setting of Section \ref{sec-orbits}. 
In particular, we assume that $H$ is one of the groups (1)--(6) as in Proposition \ref{prop: determinants}. 
By Corollary \ref{coro:coprime singularities}, we see that in the cases $N=11,13$ the group $G$ is of product type. In the cases $N=10, 14, 15$ by Lemmas \ref{lem-exclusion-of-two-groups}--\ref{Z4222 orbits} we see that the group $G$ is of product type as well. It remains to deal with the cases $N=9,12$. The first case is done by Proposition \ref{prop-N=9}, and the second one by Lemmas \ref{lem-exclusion-of-two-groups}--\ref{Z4222 orbits} and Proposition \ref{prop-k3-with-orbit-4}.
\end{proof}

\section{Case $h^0(-K_X)=1$ with cyclic quotient singularities}
\label{sec-cyclic-quotient}
In this section, we assume that $X$ the non-Gorenstein locus of $X$ consists of cyclic quotient singularities. We assume that it is not only composed of singularities of type $\frac{1}{2}(1,1,1)$, since it has already been treated in Section \ref{sec-Case $h^0(-K_X)=1$}. We prove the following 

\begin{thm}
\label{thm-cyclic-quotient-sing}
Let $X$ be a $G\mathbb{Q}$-Fano threefold where $G$ is a finite abelian group. Assume that $h^0(-K_X)=1$. Also, assume that the non-Gorenstein locus of $X$ consists only of cyclic quotient singularities. Then $G$ is of product type.
\end{thm}

The next result is obtained from the table in Proposition \ref{prop: 91 poss}, using the assumption that all the non-Gorenstein points on $X$ are cyclic quotient singularities. 

\begin{proposition}
\label{prop:baskets-cyclic-sing}
Assume that $h^0(-K_X)=1$. Assume that all non-Gorenstein points of $X$ are cyclic quotient singularities.
If $G$ is not of product type, then either the basket of $X$ consists of points of type $\frac{1}{2}(1,1,1)$, or its basket is among the following possibilities.
\end{proposition}


\begingroup
\renewcommand{\arraystretch}{1.2}
\begin{longtable}{|p{0.38\textwidth}|p{0.32\textwidth}|p{0.20\textwidth}|}
\hline
\textbf{Singularities} & \textbf{Possibilities for $H$} & \textbf{Argument} \\ \hline
$2 \times \frac{1}{10}(3,7,1)$ & $\dz/8\times\dz/4\times\dz/2$ & Lemma \ref{lem: Z842-all} \\ \hline
$3 \times \frac{1}{7}(3,4,1)$ & $\dz/6\times(\dz/3)^2$ & Lemma \ref{lem-quotient-633}  \\ \hline
$2 \times \frac{1}{11}(4,7,1)$ & $\dz/8\times\dz/4\times\dz/2$ & Lemma \ref{lem: Z842-all} \\ \hline
$6 \times \frac{1}{4}(1,3,1)$ &   
$\dz/6\times(\dz/3)^2$ & 
Lemma \ref{lem-quotient-633} \\ \hline
$2 \times \frac{1}{9}(2,7,1)$ & $\dz/8\times\dz/4\times\dz/2$ & Lemma \ref{lem: Z842-all} \\ \hline
$4 \times \frac{1}{2}(1,1,1)$, $4 \times \frac{1}{3}(1,2,1)$ &   $\dz/4\times(\dz/2)^3$ & Lemma \ref{lem: Z4222-all} \\ \hline
$4 \times \frac{1}{5}(2,3,1)$ &   $\dz/8\times\dz/4\times\dz/2$, \newline $\dz/4\times(\dz/2)^3$ & Lemma \ref{lem: Z842-all}, \newline Lemma \ref{lem: Z4222-all} \\ \hline
$4 \times \frac{1}{2}(1,1,1), 4 \times \frac{1}{4}(1,3,1)$ &  $\dz/4\times(\dz/2)^3$ & Lemma \ref{lem: Z4222-all} \\ \hline
$2 \times \frac{1}{11}(3,8,1)$ &   $\dz/8\times\dz/4\times\dz/2$ & Lemma \ref{lem: Z842-all} \\ \hline
$8 \times \frac{1}{3}(1,2,1)$ & $(\dz/2)^5$, \newline $\dz/8\times\dz/4\times\dz/2$, \newline $\dz/4\times(\dz/2)^3$ & Lemma \ref{lem: Z22222-all}, \newline Lemma \ref{lem: Z842-all}, \newline Lemma \ref{lem: Z4222-all} \\ \hline
$3 \times \frac{1}{7}(2,5,1)$ & $\dz/6\times(\dz/3)^2$ & Lemma \ref{lem-quotient-633} \\ \hline
$2 \times \frac{1}{11}(2,9,1)$ & $\dz/8\times\dz/4\times\dz/2$ & Lemma \ref{lem: Z842-all} \\ \hline
$8 \times \frac{1}{2}(1,1,1)$, $4 \times \frac{1}{3}(1,2,1)$ & $\dz/4\times(\dz/2)^3$ & Lemma \ref{lem: Z4222-all} \\ \hline
\end{longtable}
\endgroup
\begin{center}
\emph{Table 10. Possible baskets of cyclic quotient singularities}
\end{center}

\begin{lem}
\label{lem-quotient-633}
	If $H$ is isomorphic to $\dz/6\times(\dz/3)^2$, then $G$ is of product type.
\end{lem}
\begin{proof}
	Proposition \ref{prop:baskets-cyclic-sing} leaves us with the following possibilities for the basket of $X$:
	\begin{enumerate}
		\item $3\times\frac{1}{7}(3,4,1)$,
		\item $6\times\frac{1}{4}(1,3,1)$,
		\item $3\times\frac{1}{7}(2,5,1)$.
	\end{enumerate}
By Corollary \ref{cor-s-s-tilde} either $G$ is of product type, or $S$ has singularities (corresponding to non-Gorenstein points on $X$) of type $A_6$, $A_3$ and $A_6$, respectively, in each of the three cases. 
However, it contradicts Corollary \ref{cor-rank-2-is-transitive}. This contradiction shows that $G$ is of product type.
\end{proof}

\begin{lem}
\label{lem: Z22222-all}
If $H=(\dz/2)^5$, then $G$ is of product type.
\end{lem}
\begin{proof}
By Proposition \ref{prop:baskets-cyclic-sing}, either $G$ is of product type, or the basket of singularities of $X$ is $8\times\frac{1}{3}(1,2,1)$. By Lemma \ref{Z22222 orbits}, the $8$ non-Gorenstein singular points of $X$ belong to one $G$-orbit, and hence to one $H$-orbit as well. 
By Corollary \ref{cor-s-s-tilde}, we may assume that on $S$ these singular points have type $A_2$.       
Let $x$ be such a point.       
By \eqref{diagram-terminal-flag}, there exists the following diagram:
\begin{equation}
\begin{tikzcd}
(\widetilde{x}\in \widetilde{X}) \simeq (0\in \mathbb{C}^3) \ar[r, "\pi"] & X\ni x \\
(\widetilde{x}\in \widetilde{S})  \simeq (0\in \mathbb{C}^2) \ar[u, hook] \ar[r, "\pi|_{\widetilde{S}}"] & S\ni x  \ar[u, hook]
\end{tikzcd}
\end{equation}
There is an induced diagram, cf. \eqref{diagram-two-levels-of-groups}:
\begin{equation}
\begin{tikzcd}
\label{diagram-two-levels-of-groups-2}
0 \ar[r] & \mathbb{Z}/3 \ar[r] & \widetilde{G}_x \ar[r] \ar[d] & G_x \ar[r] \ar[d] & 0 \\
0 \ar[r] & \mathbb{Z}/3 \ar[u, equal] \ar[r]  & \widetilde{H}_x  \ar[r] & H_x \ar[r] & 0
\end{tikzcd}
\end{equation}

Now let $\Sigma = \{x=x_1,\ldots,x_8\}$ be a $G$-orbit (or, equivalently, an $H$-orbit) of $x$. Consider the following exact sequence 
\begin{equation}
\label{seq-x-sigma-22222}
0 \to H_x \to H \to H_\Sigma \to 0,
\end{equation}
where $H_\Sigma$ is the image of $H$ in the group of permutations of $\Sigma$. Since the action of $H$ on $\Sigma$ is transitive, we have   $H_\Sigma=(\mathbb{Z}/2)^3$. 
We obtain $H_x=(\mathbb{Z}/2)^2$. In particular, \eqref{seq-x-sigma-22222} splits: $H=H_x\times H_\Sigma$. 

By Corollary \ref{cor-max-group-direct-product}, we have $H=H_s\times H_{ns}$,  
where $H_s=(\mathbb{Z}/2)^4$ is a subgroup that acts on the minimal resolution $S'$ of $S$ symplectically, and $H_{ns}=\mathbb{Z}/2$. 
Let $\alpha$ be a non-trivial element in the kernel of the induced map $(\mathbb{Z}/2)^2=H_x\to H\to H_{ns}=\mathbb{Z}/2$. Then $\alpha\in H_x$ acts symplectically on $S'$. 
In other words, $\alpha$ is a Nikulin involution. We know that it has exactly $8$ fixed points, see Remark \ref{rem-nikulin-involution}. We claim that $\alpha$ interchanges two $(-2)$-curves over each point $x_i$ in the minimal resolution of $S$. Indeed, otherwise $\alpha$ would have more than $8$ fixed points, which is a contradiction. 

Consider the following diagram:
\begin{equation}
\label{diagram-a2-singularity}
\begin{tikzcd}
\widetilde{E}',\widetilde{E_1},\widetilde{E_2}\subset \widetilde{\mathrm{Bl}_0\mathbb{C}^2} \ar[d] \ar[r, "3:1"] & \widetilde{\mathrm{Bl}_x{S}} \supset E,E'_1,E'_2 \ar[d]  \\
\widetilde{E}\subset \mathrm{Bl}_0\mathbb{C}^2 \ar[d] & \mathrm{Bl}_x{S} \supset E_1, E_2 \ar[d]  \\
0\in \mathbb{C}^2 \ar[r, "3:1"] & S \ni x
\end{tikzcd}
\end{equation}
where horizontal arrows are quotient maps by the action of $\mathbb{Z}/3$, and 
\begin{enumerate}
\item
$\mathrm{Bl}_x{S}$ is the blow up of $S$ at the closed point $x$ with the exceptional $(-2)$-curves $E_1$ and~$E_2$, 
\item
$\mathrm{Bl}_0\mathbb{C}^2$ is the blow up of $\mathbb{C}^2$ at the closed point $0$ with the exceptional $(-1)$-curve $\widetilde{E}$,
\item
$\widetilde{\mathrm{Bl}_0\mathbb{C}^2}$ is the blow up of two $\mathbb{Z}/3$-fixed points on $\widetilde{E}$, $\widetilde{E_1}$ and $\widetilde{E_2}$ are $(-1)$-curves, and $\widetilde{E}'$ is a smooth rational $(-3)$-curve,
\item
$\widetilde{\mathrm{Bl}_x{S}}$ is the blow up of the intersection point of $E_1$ with $E_2$, so $E'$ is a $(-1)$-curve, $E_1'$ and $E_2'$ are smooth rational $(-3)$-curves. 
\end{enumerate}
Since $\alpha$ is symplectic, it lifts to an element that acts on $\mathbb{C}^2$ via the matrix $\begin{pmatrix}-1&0\\0&-1\end{pmatrix}$. However, in this case the lift of $\alpha$ to $\widetilde{\mathrm{Bl}_0\mathbb{C}^2}$ does not interchange $\widetilde{E_1}$ and $\widetilde{E_2}$, hence it does not interchange $E_1$ with $E_2$. This contradiction shows that $G$ is of product type. 
\end{proof}

\begin{lem}\label{lem: Z842-all}
	If $H$ is isomorphic to $\dz/8\times\dz/4\times\dz/2$, then $G$ is of product type.
\end{lem}
\begin{proof}
Proposition \ref{prop:baskets-cyclic-sing} leaves us with the following possibilities for the basket of $X$:
	\begin{enumerate}
		\item $2\times\frac{1}{10}(3,7,1)$,
		\item $2\times\frac{1}{11}(4,7,1)$,
		\item $2\times\frac{1}{9}(2,7,1)$,
		\item $4\times\frac{1}{5}(2,3,1)$,
		\item $2\times\frac{1}{11}(3,8,1)$,
		\item $8\times\frac{1}{3}(1,2,1)$,
		\item $2\times\frac{1}{11}(2,9,1)$.
	\end{enumerate}
By Corollary \ref{cor-s-s-tilde}, we may assume that $S$ has $k$ singular points of type $A_{r_i-1}$ that correspond to $k$ quotient singularities of index $r_i$ in the list above. By Corollary \ref{cor-rank-2-is-transitive} we see that $r_i$ could be equal only to $2$ or $3$. This leaves us with the only one case $8\times\frac{1}{3}(1,2,1)$.  
Thus $S$ has $8$ singular points of type~$A_2$. 

By Proposition \ref{prop-lift-of-2-group-is-abelian} we know that since $r=3$, the lifting $\widetilde{G}_x$ is abelian. Hence, the lifting $\widetilde{H}_x$ is abelian as well. Thus $\widetilde{H}_x$ does not interchange the two $(-2)$-curves $E_1$ and $E_2$ as in diagram \eqref{diagram-a2-singularity}. However, this contradicts Corollary \ref{cor-rank-2-is-transitive} as there are at least two $H$-orbits of $(-2)$-curves on the minimal resolution of $S$.
\end{proof}

\begin{proposition}
\label{prop-8-(1,2,1)}
Let $X$ be a $G\mathbb{Q}$-Fano threefold such that the set of non-Gorenstein singular points of $X$ is equal to $8\times\frac{1}{3}(1,2,1)$.
Assume that $S\in |-K_X|$ is a $G$-invariant K3 surface with at worst du Val singularities. Assume that 
\begin{enumerate}
\item
$H=\mathbb{Z}/4\times (\mathbb{Z}/2)^3$ where $H$ is as in \eqref{eq-exact-seq-k3-type},
\item
the group $G$ acts on the points $8\times\frac{1}{3}(1,2,1)$ transitively. 
\end{enumerate}
Then $G$ is of product type. 
\end{proposition}
\begin{proof}
Let us denote the singularities $8\times\frac{1}{3}(1,2,1)$ on $X$ by $\Sigma = \{x=x_1,\ldots,x_8\}$. By assumption, they form one $G$-orbit (or, equivalently, the $H$-orbit). 
By Corollary \ref{cor-s-s-tilde}, we may assume that they correspond to $8$ singular points on $S$ of type $A_2$.  
There exists an exact sequence 
\begin{equation}
\label{seq-x-sigma-842}
0 \to H_x \to H \to H_\Sigma \to 0,
\end{equation}
where $H_\Sigma$ is the image of $H$ in the group of permutations of $\Sigma$, and $H_x$ is the stabilizer of $x$. 
By Proposition~\ref{prop-lift-of-2-group-is-abelian} we know that since $r=3$, the lifting $\widetilde{G}_x$ is abelian, cf. diagram \eqref{diagram-two-levels-of-groups}. 
Hence, the lifting $\widetilde{H}_x$ is abelian as well. Thus $H_x$ does not interchange two $(-2)$-curves over $x\in S$ in the minimal resolution $S'$ of $S$, cf. diagram \eqref{diagram-a2-singularity}.

For $1\leq i\leq 8$, let $E_i, E'_i$ be the $(-2)$-curves on the minimal resolution $S'$ that lie over the point $x_i$ on $S$. In particular, we have $E_i\cdot E'_i=1$, $E_i\cdot E'_j=0$ for $i\neq j$. 
Note that $H_x$ cannot contain a Nikulin involution. Indeed, $H_x$ stabilizes the curves $E_i$, $E'_i$ for $1\leq i\leq 8$, and hence such an involution would have more than $8$ fixed points, which is a contradiction according to Remark \ref{rem-nikulin-involution}. By Corollary \ref{cor-max-group-direct-product}, we have $H=H_s\times \mathbb{Z}/m$ where $H_s = (\mathbb{Z}/2)^3$ acts symplectically on the minimal resolution $S'$ of $S$, and $m=4$. 
It follows that $H=H_x\times H_\Sigma$, where $H_\Sigma=(\mathbb{Z}/2)^3$ and $H_x = \mathbb{Z}/4$. In particular, \eqref{seq-x-sigma-842} splits.


Let $\sigma$ be a generator of $H_x=\mathbb{Z}/4$. 
Note that $\sigma$ preserves the curves $E_i$, $E'_i$ for $1\leq i\leq 8$. 
Observe that $\sigma^2$ is a non-symplectic involution on $S'$. It follows that the lifting of $\sigma^2$ acts on $\mathbb{C}^2$ as in diagram \eqref{diagram-a2-singularity} via the matrix $\begin{pmatrix}-1&0\\0&1\end{pmatrix}$. 
Up to renumbering, we may assume that $\sigma^2$ has the $(-2)$-curves $E_1,\ldots, E_8$ as fixed curves and the $(-2)$-curves $E'_1,\ldots, E'_8$ are preserved by $\sigma^2$. 

Since $\sigma^2$ preserves $E'_i$ for $1\leq i\leq 8$ and does not fix it pointwisely, it follows that $H_x=\langle \sigma\rangle$ acts on $E'_i$ faithfully. Also, $\sigma^2$ has two fixed points on each $E'_i$: one is the point of intersection $E_i\cap E'_i$, and we denote the other by $p_i\in E'_i$. 
Consider the local action of $\sigma^2$ near $p_i$. Since $\sigma^2$ is non-symplectic, it is given by the matrix $\begin{pmatrix}-1&0\\0&1\end{pmatrix}$ where the first eigenvector corresponds to $E'_i$, and the second eigenvector corresponds to a $\sigma^2$-fixed curve $C$ with the property $C\cdot E'_i=1$ for $1\leq i\leq 8$. By a similar computation one checks that $C\cdot E_i=0$ for $1\leq i\leq 8$. Note that $\sigma$ preserves $C$. We use the results \cite[Theorem 1.4]{BH23} on fixed loci of non-symplectic automorphisms of order $4$, see \cite[Table 6]{BH23}. 

Assume that $\sigma$ fixes $C$ pointwisely. Then $C$ admits a faithful action of $H/H_x = (\mathbb{Z}/2)^3$. Denote by $N$ the number of $\sigma^2$-fixed irreducible curves. Then $N\geq 9$ as the curves $E_1,\ldots, E_8$, and the irreducible components of $C$ are $\sigma^2$-fixed. Observe that $N\leq 10$ by \cite[Table 6]{BH23}. 

Assume that $C$ is reducible. Since $N\leq 10$, it follows that $N=10$ and $C=C_1+C_2$. In this case $C_1$ and $C_2$ are rational, and they are interchanged by $H$. Then $S'$ has ID 0.4.0.3 as in \cite[Table 6]{BH23}. However, in this case there should be two more $\sigma$-fixed rational curves, say $C'_1$ and $C'_2$. Note that $C'_i$ are also $\sigma^2$-fixed, hence they coincide with some of the curves $E_i$. Since $H$ acts on the set $\{E_1,\ldots, E_8\}$ transitively, this cannot happen. 

Assume that $C$ is irreducible. Since $C$ admits a faithful action of $(\mathbb{Z}/2)^3$, it follows that $C$ is non-rational. However, this case is not realized by \cite[Table 6]{BH23}.

Now assume that $\sigma$ does not fix $C$ pointwisely. Hence $C$ admits a faithful action of $(\mathbb{Z}/2)^4$. Assume that $C$ is irreducible. Since $(\mathbb{Z}/2)^4$ acts on $C$ faithfully, it follows that $g(C)\geq 3$. However, this contradicts \cite[Table 6]{BH23}. 
Finally, assume that $C$ is reducible, so that $C=C_1+C_2$. If $C_1$ and $C_2$ are not permuted by $(\mathbb{Z}/2)^4$, then this group faithfully acts on a smooth rational curve, which cannot happen. Then $C_1$ and $C_2$ are permuted by $(\mathbb{Z}/2)^4$. Hence 
$(\mathbb{Z}/2)^3$ acts on each component faithfully. Again, we arrive at a contradiction, because such a group cannot faithfully act on a smooth rational curve. 
\end{proof}

\begin{lem}
\label{lem: Z4222-all}
	If $H$ is isomorphic to $\dz/4\times(\dz/2)^3$, then $G$ is of product type.
\end{lem}
\begin{proof}
Proposition \ref{prop:baskets-cyclic-sing} leaves us with the following possibilities for the basket of $X$:	
\begin{enumerate}
		\item $4\times\frac{1}{2}(1,1,1),4\times\frac{1}{3}(1,2,1)$,
		\item $4\times\frac{1}{5}(2,3,1)$,
		\item $4\times\frac{1}{2}(1,1,1),4\times\frac{1}{4}(1,3,1)$,
		\item $8\times\frac{1}{3}(1,2,1)$,
		\item $8\times\frac{1}{2}(1,1,1),4\times\frac{1}{3}(1,2,1)$.
\end{enumerate}
	
Using Proposition \ref{prop-k3-with-orbit-4}, we exclude the cases (1) and (3) which leaves us with the cases (2), (4) and (5).
Consider the case $8\times\frac{1}{2}(1,1,1),4\times\frac{1}{3}(1,2,1)$. Consider the orbit $4\times\frac{1}{3}(1,2,1)$. The stabilizer of a point $x$ from this orbit $H_x$ should contain a Nikulin involution acting on the minimal resolution $S'$ of $S$. By Proposition \ref{prop-lift-of-2-group-is-abelian} the lifting $\widetilde{H}_x$ is abelian. It follows that two $(-2)$-curves over $x$ on the minimal resolution $S'$ of $S$ are not interchanged. Hence such a Nikulin involution would stabilize $8$ smooth rational curves, so it has more than $8$ fixed points, which is a contradiction, cf. Remark \ref{rem-nikulin-involution}. 
In the case $4\times\frac{1}{5}(2,3,1)$, the same argument applies. 
Finally, the case $8\times\frac{1}{3}(1,2,1)$ follows from  Proposition \ref{prop-8-(1,2,1)}. 
\end{proof}

\begin{proof}[Proof of Theorem \ref{thm-cyclic-quotient-sing}]
Follows from Proposition \ref{prop:baskets-cyclic-sing}, Lemma \ref{lem-quotient-633}, Lemma \ref{lem: Z22222-all}, Lemma \ref{lem: Z842-all}, Proposition \ref{prop-8-(1,2,1)} and Lemma \ref{lem: Z4222-all}. 
\end{proof}

\section{Case $h^0(-K_X)=1$ with terminal singularities}
\label{sec-terminal-sing}
In this section, we prove the following.

\begin{thm}
\label{thm-terminal-sing}
Let $X$ be a $G\mathbb{Q}$-Fano threefold where $G$ is a finite abelian group. Assume that $h^0(-K_X)=1$. Then $G$ is of product type.
\end{thm}

We start with the following special case. 

\begin{proposition}
\label{prop-term-index-2}
Assume that all the non-Gorenstein points on $X$ have index $2$. Then $G$ is of product type.
\end{proposition}
\begin{proof}
We work in the setting of Section \ref{sec-orbits}. 
In particular, we assume that $H$ is one of the groups (1)--(6) as in Proposition \ref{prop: determinants}. 
Lemma \ref{lem-exclusion-of-two-groups} excludes two of these groups, leaving us with only $4$ possibilities. 

The case when all the non-Gorenstein points singularities are cyclic quotient singularities of index $2$ is already considered in Theorem \ref{thm-half-points}. 
Hence we may assume that at least one point, say $P_1\in X$, is not a cyclic quotient singularity. 
By Section \ref{sec-about-baskets}, in the basket of $P_1$ there are at least $2$ cyclic quotient singularities. 


By Lemma \ref{Z22222 orbits} and Lemma \ref{Z4222 orbits} we see that the groups $H=(\mathbb{Z}/2)^5$ and $H=\mathbb{Z}/4\times(\mathbb{Z}/2)^3$ are not realized, since in this case $t_1$ is divisible by $4$. Hence 
the total number $N$ of half-points in the baskets of $k_1\times P_1$ is divisible by $8$. However, by Proposition \ref{half} we know that $9\leq N\leq 15$. This is a contradiction. 

Consider the case $H=\mathbb{Z}/6\times (\mathbb{Z}/3)^2$. By Lemma \ref{Z633 orbits} we have that $k_1$ is divisible by~$3$. Hence 
the total number $N$ of half-points in the baskets of $k_1\times P_1$ is divisible by $6$. By Proposition \ref{half} we know that $9\leq N\leq 15$. Thus $N=12$ and so $k_1=6$.  By Lemma~\ref{Z633 orbits} we see that $X$ has no other non-Gorenstein singular points. Also, by Corollary \ref{cor-rank-2-is-transitive}, the surface $S$ has only singular points of type $A_1$ or $A_2$. However, this contradicts Proposition \ref{prop-non-gorenstein-exclusion}.


Consider the case $H=\mathbb{Z}/8\times \mathbb{Z}/4\times \mathbb{Z}/2$. By Lemma \ref{Z842 orbits} we have that $k_1$ is divisible by $2$. Hence 
the total number $N$ of half-points in the baskets of $k_1\times P_1$ is divisible by $4$. However, by Proposition \ref{half} we know that $9\leq N\leq 15$. Thus $N=12$, and so $k_1=6$. However, in this case $G$ has at least two orbits of singular points, which contradicts Lemma~\ref{Z842 orbits}. 
\end{proof}

Using Proposition \ref{prop-term-index-2}, Proposition~\ref{prop: 91 poss} and Corollary \ref{coro:coprime singularities}, we obtain the following list of remaining possibilities. 

\begingroup
\renewcommand{\arraystretch}{1.15} 
\begin{longtable}{|p{0.30\textwidth}|p{0.35\textwidth}|p{0.25\textwidth}|p{0.05\textwidth}|}
\hline
\textbf{Singularities of $X$} & \textbf{Basket of $X$} & \textbf{Possibilities for $H$} \\ \hline
\endhead
$2\times cA/4$ or \newline $3\times cA/4$  & $6 \times \frac{1}{4}(1,3,1)$ &   $\dz/8\times\dz/4\times\dz/2$, \newline $\dz/6\times(\dz/3)^2$ \\ \hline
$2\times cA/4$ & $4 \times \frac{1}{5}(2,3,1)$ &   $\dz/8\times\dz/4\times\dz/2$ \\ \hline
 $4\times cA/3$ or \newline $4\times cD/3$ & $8 \times  \frac{1}{3}(1,2,1)$ &   $\dz/8\times\dz/4\times\dz/2$, \newline 
$\dz/4\times(\dz/2)^3$ \\ \hline
$2\times cAx/4$ & $8 \times \frac{1}{2}(1,1,1)$, $2 \times \frac{1}{4}(1,3,1)$ & $\dz/8\times\dz/4\times\dz/2$ \\ \hline
$2\times cAx/4$ & $10 \times \frac{1}{2}(1,1,1)$, $2 \times \frac{1}{4}(1,3,1)$ & $\dz/8\times\dz/4\times \dz/2$ \\ \hline
$4\times cA/2,4\times \frac{1}{3}(1,2,1)$  & $8 \times \frac{1}{2}(1,1,1)$, $4 \times \frac{1}{3}(1,2,1)$ & $\dz/4\times(\dz/2)^3$ \\ \hline
$2\times cAx/4$ & $6 \times \frac{1}{2}(1,1,1), 2 \times \frac{1}{4}(1,3,1)$  & $\dz/8\times\dz/4\times\dz/2$ \\ \hline
$4\times cAx/4$ & $4 \times \frac{1}{2}(1,1,1), 4 \times \frac{1}{4}(1,3,1)$ & $\dz/8\times\dz/4\times\dz/2$ \\ \hline
$3\times cAx/4$ & $6 \times \frac{1}{2}(1,1,1), 3 \times \frac{1}{4}(1,3,1)$ & $\dz/6\times(\dz/3)^2$ \\ \hline
\end{longtable}
\begin{center}
\emph{Table 11. Possible baskets of singularities on $X$} 
\end{center}
\endgroup
\begin{proof}[Proof of Theorem \ref{thm-terminal-sing}]
Proposition \ref{prop-non-gorenstein-exclusion} together with Corollary \ref{cor-rank-2-is-transitive} 
exclude all cases except the following:
\begin{enumerate}
    \item $4\times cA/3$, $H=\dz/4\times(\dz/2)^3$,
    \item $4\times cD/3$, $H=\dz/4\times(\dz/2)^3$,
    \item $4\times cA/2+4\times \frac{1}{3}(1,2,1)$, $H=\dz/4\times(\dz/2)^3$.
\end{enumerate}

In the case (3), by Corollary \ref{cor-s-s-tilde} and  Proposition \ref{prop-non-gorenstein-exclusion} we may assume that we have $4A_2$ singularities on $S$ that correspond to $4\times \frac{1}{3}(1,1,1)$, and $4$ du Val singularities of type different from $A_1$ or $A_2$. Denote by $f\colon S'\to S$ the minimal resolution of $S$. It follows that there are at least $20$ $f$-exceptional $(-2)$-curves. However, this contradicts the fact that on a smooth K3 surface $S'$ one has $\rho(S')\leq 20$. Hence this case is not realized. 

Consider the cases (1) (or (2)), that is, when $X$ has singularities $4\times cA/3$ (or $4\times cD/3$), and $H=\dz/4\times(\dz/2)^3$. By Proposition \ref{prop-non-gorenstein-exclusion} and Remark \ref{rem-index-divides-exponent}, we see that $S$ has $4$ singularities of type $A_k$ where $k\geq 5$. However, arguing as in the previous case we get a contradiction with $\rho(S')\leq 20$. Hence this case is not realized either. 
The proof is completed.
\end{proof}

\section{Proof of main results}
\label{proof-main-results}
\begin{proof}[Proof of Theorem \ref{thm-main-thm}]
Assume that $G$ is a group that faithfully acts on a rational connected threefold $X$. By a standard argument, we may assume that $X$ is a projective $G\mathbb{Q}$-Mori fiber space over the base $Z$. If $\dim Z>0$ then $G$ is of product type by \cite[Corollary 3.17]{Lo24}. Hence we may assume that $X$ is a $G\mathbb{Q}$-Fano threefold. 

If $h^0(-K_X)=0$ and for any $G\mathbb{Q}$-Fano threefold $X'$ which is $G$-birational to $X$ we have $h^0(-K_{X'})=0$, then $G$ is of type (3) as in Theorem \ref{thm-main-thm}. Hence we may assume that $h^0(-K_X)>0$. 
Also, by the proof of \cite[Theorem 1.7]{Lo24} we may assume that for any $G$-invariant element $S\in |-K_X|$, the pair $(X, S)$ is plt, so $S$ is a K3 surface with at worst du Val singularities. 
If $h^0(-K_X)\geq 2$ then by Theorem \ref{thm-h0-geq-2} we have that $G$ is of product type. 

So we may assume that $h^0(-K_X)=1$. In particular, this implies that the set of non-Gorenstein singularities of $X$ is non-empty. 
If the set of non-Gorenstein singularities consists of points of type $\frac{1}{2}(1,1,1)$, then by 
Theorem~\ref{thm-half-points} we have that $G$ is of product type. 
If the set of non-Gorenstein singularities consists of cyclic quotient singularities, then by Theorem \ref{thm-cyclic-quotient-sing} we have that $G$ is of product type. 
Finally, if the set of non-Gorenstein singularities consists of terminal points which are non necessarily cyclic quotient singularities, then by 
Theorem \ref{thm-terminal-sing} we have that $G$ is of product type.  
\end{proof}

\begin{proof}[Proof of Proposition \ref{prop-intro-k3-groups}]
Follows from Theorem \ref{thm-maximal-k3-groups} and Remark \ref{rem-special-K3-groups-are-PT}.
\end{proof}

\section{Appendix: Intersection matrices for K3 surfaces acted on by $(\dz/2)^5$ and $\dz/4\times(\dz/2)^3$}
\label{sec-appendix}
\begin{proposition}\label{prop: matrices for Z4333}
    Let $S$ be a $K3$ surface with a faithful action of $H=\dz/4\times(\dz/2)^3$. Then the intersection matrix $M$ on $\pic^H(S)$ is one of the following.
    \begin{enumerate}

    \item
    \[
    \begin{pmatrix}
    0 & 4 \\
    4 & 0
    \end{pmatrix},
    \]
    
    \item
    \[
    \begin{pmatrix}
    0 & 2 \\
    2 & 0
    \end{pmatrix},
    \]
    
    \item
    \[
    \begin{pmatrix}
    0 & 0 & 2 \\
    0 & -8 & 0 \\
    2 & 0 & 0
    \end{pmatrix},
    \]
    
    \item
    \[
    \begin{pmatrix}
    0 & 0 & 0 & 2 \\
    0 & -4 & 0 & 0 \\
    0 & 0 & -4 & 0 \\
    2 & 0 & 0 & 0
    \end{pmatrix},
    \]
    
    \item
    \[
    \begin{pmatrix}
    0 & 0 & 0 & 0 & 0 & 2 \\
    0 & -4 & -2 & 2 & 2 & 0 \\
    0 & -2 & -4 & 2 & 2 & 0 \\
    0 & 2 & 2 & -4 & 0 & 0 \\
    0 & 2 & 2 & 0 & -4 & 0 \\
    2 & 0 & 0 & 0 & 0 & 0
    \end{pmatrix},
    \]
    
    \item
    \[
    \begin{pmatrix}
    0 & -2 & -2 & 2 & -2 & 0 \\
    -2 & -4 & -2 & 0 & -2 & -2 \\
    -2 & -2 & -4 & 0 & 0 & 0 \\
    2 & 0 & 0 & 0 & 0 & 0 \\
    -2 & -2 & 0 & 0 & -4 & 0 \\
    0 & -2 & 0 & 0 & 0 & -4
    \end{pmatrix},
    \]
    
    \item
    \[
    \begin{pmatrix}
    0 & 0 & 0 & 2 & 0 & 0 \\
    0 & -4 & -2 & 0 & -2 & -2 \\
    0 & -2 & -4 & 0 & 0 & 0 \\
    2 & 0 & 0 & 0 & 0 & 0 \\
    0 & -2 & 0 & 0 & -4 & 0 \\
    0 & -2 & 0 & 0 & 0 & -4
    \end{pmatrix},
    \]
    
    \item
    \[
    \begin{pmatrix}
    0 & 0 & 2 & 0 & 0 & 0 \\
    0 & -4 & 0 & -2 & -6 & 2 \\
    2 & 0 & 0 & 0 & 0 & 0 \\
    0 & -2 & 0 & -4 & -6 & 2 \\
    0 & -6 & 0 & -6 & -20 & 8 \\
    0 & 2 & 0 & 2 & 8 & -4
    \end{pmatrix},
    \quad
    \]
    
    \item
    \[
    \begin{pmatrix}
    0 & -2 & 2 & -4 & -4 & 2 \\
    -2 & -4 & 0 & -4 & -6 & 2 \\
    2 & 0 & 0 & 0 & 0 & 0 \\
    -4 & -4 & 0 & -8 & -8 & 4 \\
    -4 & -6 & 0 & -8 & -12 & 4 \\
    2 & 2 & 0 & 4 & 4 & -4
    \end{pmatrix},
    \]
    
    \item
    \[
    \begin{pmatrix}
    0 & 0 & -2 & -4 & -2 & -6 \\
    0 & -8 & -2 & -20 & -14 & -26 \\
    -2 & -2 & -4 & -8 & -4 & -12 \\
    -4 & -20 & -8 & -60 & -40 & -78 \\
    -2 & -14 & -4 & -40 & -28 & -52 \\
    -6 & -26 & -12 & -78 & -52 & -104
    \end{pmatrix}.
    \]

\end{enumerate}
\end{proposition}

\begin{proposition}\label{prop: matrices for Z22222}
    Let $S$ be a $K3$ surface with a faithful action of $H=(\dz/2)^5$. Then the intersection matrix $M$ on $\pic^H(S)$ is one of the following.
\begin{enumerate}
    \item
    \[
    \begin{pmatrix} 8 \end{pmatrix},
    \]

    \item
    \[
    \begin{pmatrix} 0 & 2 \\ 2 & 0 \end{pmatrix},
    \]

    \item
    \[
    \begin{pmatrix} 0 & 4 \\ 4 & 0 \end{pmatrix},
    \]

    \item
    \[
    \begin{pmatrix} 0 & -2 \\ -2 & 0 \end{pmatrix},
    \]

    \item
    \[
    \begin{pmatrix} 0 & 0 & 2 \\ 0 & -8 & 0 \\ 2 & 0 & 0 \end{pmatrix},
    \]

    \item
    \[
    \begin{pmatrix} 0 & 0 & 0 & 2 \\ 0 & -4 & 0 & 0 \\ 0 & 0 & -4 & 0 \\ 2 & 0 & 0 & 0 \end{pmatrix},
    \]

    \item
    \[
    \begin{pmatrix}
    0 & 0 & 0 & 0 & 2 \\
    0 & -4 & -2 & -2 & 0 \\
    0 & -2 & -4 & -2 & 0 \\
    0 & -2 & -2 & -4 & 0 \\
    2 & 0 & 0 & 0 & 0
    \end{pmatrix},
    \]

    \item
    \[
    \begin{pmatrix}
    -4 & -4 & 0 & -2 & -2 \\
    -4 & -8 & 0 & 0 & 0 \\
    0 & 0 & 0 & -2 & 0 \\
    -2 & 0 & -2 & -4 & 0 \\
    -2 & 0 & 0 & 0 & -4
    \end{pmatrix},
    \]

    \item
    \[
    \begin{pmatrix}
    -4 & 2 & -4 & -6 & 2 \\
    2 & 0 & 0 & 0 & 0 \\
    -4 & 0 & -4 & -2 & 0 \\
    -6 & 0 & -2 & -12 & 6 \\
    2 & 0 & 0 & 6 & -4
    \end{pmatrix},
    \]

    \item
    \[
    \begin{pmatrix}
    -52 & -10 & -20 & -38 & -28 \\
    -10 & -4 & -4 & -6 & -4 \\
    -20 & -4 & -8 & -14 & -10 \\
    -38 & -6 & -14 & -28 & -20 \\
    -28 & -4 & -10 & -20 & -16
    \end{pmatrix},
    \]

    \item
    \[
    \begin{pmatrix}
    -4 & 0 & -2 & -4 & 0 \\
    0 & 0 & 0 & 0 & 2 \\
    -2 & 0 & -4 & 0 & 0 \\
    -4 & 0 & 0 & -8 & 0 \\
    0 & 2 & 0 & 0 & 0
    \end{pmatrix},
    \]

    \item
    \[
    \begin{pmatrix}
    0 & 0 & 2 & 0 & 0 \\
    0 & -4 & 0 & -4 & -2 \\
    2 & 0 & 0 & 0 & 0 \\
    0 & -4 & 0 & -12 & -6 \\
    0 & -2 & 0 & -6 & -4
    \end{pmatrix},
    \quad
    \]

    \item
    \[
    \begin{pmatrix}
    0 & 2 & 2 & -2 & 0 \\
    2 & 4 & 4 & -2 & -2 \\
    2 & 4 & 0 & 0 & 0 \\
    -2 & -2 & 0 & -4 & 0 \\
    0 & -2 & 0 & 0 & -4
    \end{pmatrix},
    \]

    \item
    \[
    \begin{pmatrix}
    -4 & -4 & -2 & -2 & 0 \\
    -4 & -8 & 0 & 0 & 0 \\
    -2 & 0 & -4 & 0 & 0 \\
    -2 & 0 & 0 & -4 & 2 \\
    0 & 0 & 0 & 2 & -4
    \end{pmatrix}.
    \]

\end{enumerate}
\end{proposition}

\end{document}